\documentclass{amsart}          

\usepackage{amsmath,amsthm}     
\usepackage{amssymb,bbm}            
\usepackage{euscript}           
\usepackage{array,enumerate,calc,graphicx,xcolor}
\usepackage[matrix,arrow,curve,frame]{xy}    
\usepackage[colorlinks]{hyperref}
\usepackage{stmaryrd}
\usepackage{tikz-cd}

\makeatletter
\newcommand*\bigcdot{\mathpalette\bigcdot@{.7}}
\newcommand*\bigcdot@[2]{\mathbin{\vcenter{\hbox{\scalebox{#2}{$\m@th#1\bullet$}}}}}
\makeatother

\usepackage{tikz}
\usetikzlibrary{matrix,arrows,decorations.pathmorphing}

\xymatrixcolsep{1.9pc}                          
\xymatrixrowsep{1.9pc}
\newdir{ >}{{}*!/-5pt/\dir{>}}                  

\newtheorem{defn}[subsection]{Definition}
\newtheorem{prop}[subsection]{Proposition}

\newtheorem{cor}[subsection]{Corollary}
\newtheorem{lemma}[subsection]{Lemma}

\theoremstyle{definition}  
\newtheorem{example}[subsection]{Example}

\newtheorem{remark}[subsection]{Remark}

  {\end{list}}

\newcommand{\dfn}{\textbf} 

\newcommand{\mdfn}[1]{\dfn{\mathversion{bold}#1}} 

\newcommand{\tens}              {\otimes}               

\newcommand{\iso}               {\cong}

\newcommand{\cat}{\EuScript}    
\newcommand{\cB}{{\cat B}}      

\newcommand{\cF}{{\cat F}}

\newcommand{\Mod}{\text{Mod}}

\newcommand{\cBZ}{\cB\Z}

\newcommand{\field}[1]  {\mathbb #1} 

\newcommand{\Z}         {\field Z}
\newcommand{\C}         {\field C}
\newcommand{\Q}         {\field Q}

\DeclareMathOperator*{\Ann}{Ann}

\DeclareMathOperator{\Hom}{Hom}
\DeclareMathOperator{\uHom}{\underline{Hom}}
\DeclareMathOperator{\uExt}{\underline{Ext}}
\DeclareMathOperator{\mExt}{\uExt}
\DeclareMathOperator{\uTor}{\underline{Tor}}
\DeclareMathOperator{\Ext}{Ext}
\DeclareMathOperator{\Tor}{Tor}

\DeclareMathOperator{\coeq}{coeq}

\DeclareMathOperator{\eq}{eq}

\newcommand{\ra}{\rightarrow}                   
\newcommand{\lra}{\longrightarrow}              
\newcommand{\la}{\leftarrow}                    
\newcommand{\lla}{\longleftarrow}               
\newcommand{\llra}[1]{\stackrel{#1}{\lra}}      
\newcommand{\llla}[1]{\stackrel{#1}{\lla}}      

\newcommand{\fib}{\twoheadrightarrow}           

\newcommand{\inc}{\hookrightarrow}              
\newcommand{\dbra}{\rightrightarrows}           

\newcommand{\blank}{-}                          
\newcommand{\id}{id}                            
\newcommand{\und}{\underline}

\newcommand{\adjoint}{\rightleftarrows}

\newcommand{\pt}{pt}

\numberwithin{equation}{section}

\newcommand{\Ab}{\cat Ab}

\newcommand{\ev}{ev}

\newcommand{\bbox}{\Box}

\newcommand{\MMack}[2]{\xymatrix{
{#1} \ar@(ul,dl)[] \ar@<0.5ex>[r] & {#2}\ar@<0.5ex>[l]}}

\newcommand{\MMod}{\!-\!\Mod}
\DeclareMathOperator{\coker}{coker}

\newcommand{\mZ}{\underline{\smash{\Z}}}
\newcommand{\mQ}{\underline{\smash{\Q}}}
\newcommand{\mA}{\underline{\smash{A}}}

\DeclareMathOperator{\stab}{stab}

\DeclareMathOperator{\Res}{Res}
\DeclareMathOperator{\Ind}{Ind}
\DeclareMathOperator{\coInd}{coInd}
\newcommand{\res}{\mathord\downarrow}
\newcommand{\ind}{\mathord\uparrow}

\DeclareMathOperator{\FP}{FP}
\DeclareMathOperator{\lcm}{lcm}

\newcommand{\uH}{\underline{H}}
\newcommand{\Or}{{\cat Or}}

\DeclareMathOperator{\Pre}{Pr}
\newcommand{\mInd}{\underline{\Ind}}
\newcommand{\mRes}{\underline{\Res}}

\numberwithin{equation}{subsection}

\newcommand{\cln}{\colon\!}

\begin{document}
\title{Mackey homological algebra over cyclic groups}

\author{Daniel Dugger}
\address{Department of Mathematics, University of Oregon, Eugene, OR 97403}
\author{Christy Hazel}
\address{Department of Mathematics, Grinnell College, Grinnell, IA 50112}

\begin{abstract}
Let $C_n$ denote a cyclic group of order $n$. In this paper we investigate modules and chain complexes over the constant integral Mackey functor $\mZ$ and perform some related homological calculations.   Along the way we develop a number of foundational tools for working with these categories.  These results are useful for the study of $RO(C_n)$-graded Bredon cohomology, though such applications are delegated to a sequel paper. 
\end{abstract}

\maketitle

\tableofcontents

\section{Introduction}
Let $G$ be a finite group and let $\mZ_G$ (or just $\mZ$) denote the constant-coefficient Mackey functor with value $\Z$.  Given a collection $\cF$ of subgroups of $G$, let $I_\cF\subseteq \mZ$ be the ideal generated by the elements $1_{G/H}\in \mZ(G/H)$ for $H\in \cF$.  Given two such collections $\cF_1$ and $\cF_2$, computing (and understanding) the Mackey functors $\uTor_*^{\mZ}(\mZ/I_{\cF_1},\mZ/I_{\cF_2})$ and $\uExt^*_{\mZ}(\mZ/I_{\cF_1},\mZ/I_{\cF_2})$ is a fundamental problem in $G$-equivariant homological algebra that is intimately connected to the problem of computing the $RO(G)$-graded Bredon cohomology groups $H^\star_G(\pt;\mZ)$.  This perspective was already clear in \cite{Z}, where homological computations were done for cyclic groups $C_{\ell^e}$ of prime-power order and where the cohomology ring $H^\star_{C_{\ell^2}}(\pt;\mZ)$ was computed.  The goal of the present paper is to study the case of a general cyclic group $C_n$.  

As one might guess, the ``answers'' for $C_n$ are in some sense completely determined by the answers for the prime-power case.  
However, turning this philosophy into precise statements takes some work
(for example, it is certainly not the case that the category of $\mZ_{C_n}$-modules decomposes as a product of the categories of $\mZ_{C_{\ell^e}}$-modules, or anything as simple as that). And in fact, such decompositional approaches are not necessarily the best approach: as a classical example, the group cohomology of a cyclic group is much easier to describe as a global object than by patching it together from its local pieces.  We will find that once the necessary machinery is in place many computations for $C_n$ can be done just as easily as the analogous computations for prime powers---one just has to know the ropes.

The goal of this paper is to develop the basic tools needed to facilitate working with the category of $\mZ$-modules for $G=C_n$, and then we explore some of the homological algebra of this category.  The end product is a hodgepodge of foundational results---a pleasant stroll through the woods, if you like, rather than a hike up to a noteworthy viewpoint.  Section~\ref{se:appl} gives a small taste of applications, though.  In a sequel paper 
we use these tools to give a very simple computation of the  regular portion of the Bredon cohomology ring $H^\star_{C_n}(\pt;\mZ)$, as well as significant information about the irregular portion.  The computations for the so-called `positive cone' have already been done in \cite{BD}, but the algebraic approach developed here is independent and a useful complement to their methods. 

From one perspective, the results of this paper are not breaking new ground.  Homological computations with Mackey functors are a topic of much study---we mention \cite{B2}, \cite{BSW}, \cite{MQS}, and \cite{Z} as just a small selection.  We particularly acknowledge the unpublished preprint \cite{Z}, as it contains a number of important ideas that anticipate our current work.  
Nevertheless, for a topologist who wants to get started in this subject and make explicit computations, the existing papers are more than a little formidable.  We hope the present paper can help serve as an introduction to this subject,
and that readers will come away with a
concrete set of techniques that will make these computations easier and more accessible than they currently are.  \bigskip

Let $C_n$ denote the cyclic group of order $n$ with generator $t$. For $d|n$ let $\cF_d$ be the collection of all subgroups whose order divides $d$.  Write $I_d=I_{\cF_d}$ for short.  The following computations are examples of the ones carried out in Section~\ref{se:cyclic}.

\begin{prop} 
\label{pr:intro-computations}
In the category of $\mZ_{C_n}$-modules one has
\begin{enumerate}[(a)]
\item $\uTor_i(\mZ/I_a,\mZ/I_b)\iso 
\begin{cases}
\mZ/I_{(a,b)} & \text{if $i=0$ or $i=3$},\\
0 & \text{otherwise.}
\end{cases}
$
\item $\uExt^i(\mZ/I_a,\mZ/I_b)\iso
\begin{cases}
\mZ/I_{(a,b)} & \text{if $i=0$ or $i=3$},\\
0 & \text{otherwise.}
\end{cases}
$
\item
$\uExt^i(\mZ/I_a,\mZ)\iso
\begin{cases}
\mZ/I_{a} & \text{if $i=3$},\\
0 & \text{otherwise.}
\end{cases}
$
\end{enumerate}
\end{prop}

Note that our convention is that underlined objects are Mackey functors, so the isomorphisms in the above result are isomorphisms of $\mZ_{C_n}$-modules.  
As just one example of the utility of the above information, note that 
the abelian group $\Ext^1(\mZ/I_a,\mZ/I_b)$ is the value of $\uExt^1(\mZ/I_a,\mZ/I_b)$ at the trivial orbit.  So Proposition~\ref{pr:intro-computations}(b) implies that these $\Ext^1$ groups vanish for any $a$ and $b$.  This is very useful in spectral sequence computations.  

The above computations get one thinking that $\mZ/I_a$ is an analog of the abelian group $\Z/a$, but one should be careful here: in some ways it is, and in some ways it isn't.  The most direct analog of $\Z/a$ is the cokernel of $\mZ \llra{\cdot a} \mZ$, which we would write as $\mZ/a$.  The homological computations for those modules play out exactly the same as they do in abelian groups, with the expected $\Tor_1$ and $\Ext^1$ objects and everything vanishing in higher degrees.  The modules $\mZ/I_a$, on the other hand, when constructed in the non-equivariant context (i.e. when $n=1$) are simply the zero module: you take $\Z$ and you mod out by $1$.  So the motto to keep in mind is that in some ways $\mZ/I_a$ behaves like $\Z/a$, and in other ways it behaves like $0$.  

The nontrivial $\Ext^3(\mZ/I_a,\mZ)$ is particularly noteworthy.  The generator here gives a kind of Bockstein operator, and it turns out this plays an important role in understanding $H^\star_{C_n}(\pt)$.  This will be discussed in detail in \cite{DH2}, but we mention it here to highlight the fundamental importance of some of these computations.\bigskip  

As soon as $n>2$ the module theory of $\mZ$ is ``wild'' in a certain technical sense, loosely meaning that one does not have a manageable class of indecomposables such as one has when studying finite abelian groups.  But for a topologist studying $RO(G)$-graded Bredon cohomology, some $\mZ$-modules are more important than others.  
The first $\mZ$-modules (beyond free ones) that appear in the picture are the quotients $\mZ/I_a$ (see Section~\ref{se:Bredon} for why).  From there it is not surprising that $\uTor_*(\mZ/I_a,\mZ/I_b)$ and corresponding $\uExt$ modules start to show up, but Proposition~\ref{pr:intro-computations} says that they aren't anything new.  

With the $\mZ/I_a$ floating around it doesn't take long for the modules $I_a$ themselves to enter into computations.  Next are quotients $I_a/I_b$ and box products like $I_a\bbox I_b$ and $(\mZ/I_a)\bbox I_b$.  But it turns out that $I_a\bbox I_b\iso I_{[a,b]}$ and $(\mZ/I_a)\bbox I_b\iso I_b/I_{[a,b]}\iso I_{(a,b)}/I_a$, where $(a,b)$ is the gcd and $[a,b]$ is the lcm. 
One can keep going, introducing more and more modules via iterated extensions and box products, but as we noted at the start this process doesn't really end---the module theory is just too complex.  Still, the ones we have listed so far are a good starting point for an equivariant topologist who is building their toolkit.  

There is one more class of modules that appears fairly quickly in Bredon cohomology calculations.  If $a$ and $b$ are divisors of $n$ where $a|b$ then one can make a $\mZ$-module called $\mZ(b;a)$ whose value at every orbit is $\Z$, the action maps are all the identity, and where the restriction and transfer maps are all either the identity or multiplication by an associated index.  See Section~\ref{se:forms-of-Z} for a precise definition.  For the special case where $n$ is a prime power these $\mZ(b;a)$ modules previously appeared both in \cite{HHR2} and in \cite{Z}.  

The collection of modules $\mZ/I_a$, $I_a$, $I_a/I_b$, and $\mZ(b;a)$ are fundamental constructions that appear over and over again in computations.  One of the goals of the paper is to develop a facility for working with these.  One example of this was the isomorphism $I_a\bbox I_b\iso I_{[a,b]}$ mentioned above.  As just one other example, a certain recognition theorem for the modules $\mZ/I_a$ (see Proposition~\ref{pr:Z/F_a-identify}) turns out to be very useful.  

Finally, we remark that the first place in algebraic topology that homological algebra becomes important is in the K\"unneth and Universal Coefficient Theorems.  In the context of Bredon cohomology these are of course spectral sequences, since we learned in Proposition~\ref{pr:intro-computations} that we don't just have $\Tor_1$ and $\Ext^1$ in this setting.  In Section~\ref{se:appl} we give a quick tour of these spectral sequences with some examples.

\subsection{Organization of the paper}
Sections~\ref{se:background}, \ref{se:fam-ideals}, and \ref{se:change} are written in the context of a general finite group $G$, with the hope that this material will be useful for future research beyond the scope of the current paper.  Sections~\ref{se:cyclic} through \ref{se:appl} focus on the particular case of cyclic groups.

\subsection{Notation and terminology}
Although we do not restrict to cyclic groups until Section~\ref{se:cyclic}, these are frequently used for examples in the earlier material.  So we establish some basic notation here. For a cyclic group $C_n$ we write $t$ for a chosen generator. If $d|n$ then $C_n$ has a unique subgroup of size $d$, so we write $C_d\subseteq C_n$ to denote this subgroup.  Let $\Theta_d=C_n/C_{\frac{n}{d}}$, which is the canonical $C_n$-orbit of size $d$.  Pictorially, $\Theta_d$ is a necklace of $d$ beads where the action of $t$ rotates each bead one to its left. For any group $G$ we write $\pt$ for the orbit $G/G$. 

If $a,b\in \Z$ then $(a,b)$ is the gcd and $[a,b]$ is the lcm.

\begin{remark}
\label{re:gcd}
Starting in Section~\ref{se:cyclic} there are several places in the paper where we need exotic-looking identities between iterated lcm and gcd operations.  A way to check such things is to consider the exponents in prime factorizations, so that lcm and gcd correspond to max and min.  Such equations always come down to equations between iterated max-min operations, which can be verified by (at worst) a case argument.  It is also worth noting that max and min distribute over each other (just as AND and OR do), leading to useful formulas such as
\[ [(x,y),(x,z)]=(x,[y,z]) \quad \text{and} \quad ([x,y],[x,z])=[x,(y,z)],
\]
and the analogs of deMorgan's laws for max/min yield
\[ (\tfrac{x}{d},\tfrac{x}{e})=\tfrac{x}{[d,e]} \quad\text{and}\quad [\tfrac{x}{d},\tfrac{x}{e}]=\tfrac{x}{(d,e)}.
\]
\end{remark}

\subsection{Acknowledgments}  
Some of the work on this paper was done while the first author was a visitor at the Institute for Computational and Experimental Research in Mathematics, in Providence, RI, which is supported by NSF grant DMS-1929284.  Said author is grateful to ICERM and Brown University for providing an excellent working environment.  


\section{Background on Mackey functors and \mdfn{$\mZ$}-modules}
\label{se:background}

Let $G$ be a finite group. In this section we review the definition of $G$-Mackey functors, the box product and the internal hom, and set up the machinery needed for homological algebra.  Although this material is in some sense standard---see e.g. \cite{L1} or \cite{B1}, as two examples---carrying out computations requires us to establish some conventions very carefully.  So it seems worthwhile to provide a self-contained treatment tuned to our specific needs.  See also \cite{TW} and \cite{W} for further background on Mackey functors. \medskip

We write $\cB_G$ for the \textbf{Burnside category} of $G$.  The objects of $\cB_G$ are the finite $G$-sets and the morphisms are described in terms of spans.  Given $G$-sets $S_1$ and $S_2$ a \textbf{span} from $S_1$ to $S_2$ is a diagram of the form
\[ \xymatrixrowsep{1pc}\xymatrix{
& T\ar[dl]\ar[dr] \\
S_2 && S_1
}
\]
where $T$ is a finite $G$-set (all maps between $G$-sets are assumed to be equivariant, unless otherwise noted).  Note our ``right-to-left'' convention here. Two spans $[S_2\leftarrow T \rightarrow S_1]$ and $[S_2\leftarrow T' \rightarrow S_1]$ are considered equivalent if there exists an isomorphism $T\to T'$ making the diagram 
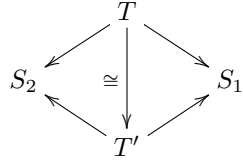
\[
\xymatrixrowsep{1pc}\xymatrix{
&T\ar[dd]_{\cong}\ar[dl]\ar[dr]&\\
S_2& & S_1\\
&T'\ar[ul]\ar[ur]&
}
\]
commute. We add two spans $[S_2\leftarrow T \rightarrow S_1]$, $[S_2\leftarrow T' \rightarrow S_1]$ using disjoint union to get $[S_2\leftarrow T\sqcup T' \rightarrow S_1]$. Note this operation respects the equivalence relation. The set of morphisms $\cB_G(S_1,S_2)$ is then defined to be the Grothendieck group of equivalence classes of spans
under the disjoint union operation.  Composition is given by pullback, that is $[S_3\leftarrow T' \rightarrow S_2] \circ [S_2\leftarrow T \rightarrow S_1]$ is the span obtained by deleting the $S_2$ in the diagram
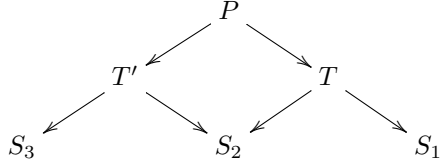
\[ 
\xymatrixrowsep{1pc}\xymatrix{
& & P\ar[dl]\ar[dr] & & \\
& T'\ar[dl] \ar[dr]& &T\ar[dl] \ar[dr] &\\
S_3 & & S_2 & & S_1
}
\]
where $P$ is the pullback of $T'\to S_2\leftarrow T$.
\smallskip

The category $\cB_G$ has a symmetric monoidal structure induced by Cartesian product of $G$-sets, with the object $\pt$ as the unit (here $\pt$ denotes a singleton).  It also comes with a duality functor $D\colon \cB_G\ra \cB_G^{op}$, which is the identity on objects and on morphisms is induced by
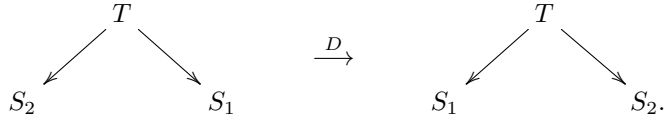
\[ \xymatrixrowsep{0.1pc}\xymatrix{
& T\ar[ddl]\ar[ddr] &&&& T \ar[ddl]\ar[ddr]
\\
&&&\llra{D} \\
S_2 && S_1 && S_1 && S_2.
} 
\]
The functor $D$ will be helpful for book-keeping purposes as we consider various contravariant and covariant functors involving $\cB_G$. 
For example there is an adjunction isomorphism
\[ \cB_G(A\times B,C)\iso \cB_G(B, A\times C)
\]
which has the evident effect on spans.  This is clearly natural in $B$ and $C$, but the naturality in $A$ is a bit tricky because the object on the left is contravariant in $A$ and the  one on the right is covariant in $A$.  The best way to write the formula is
\[ \cB_G(A\times B,C)\iso \cB_G(B,DA\times C),
\]
as this gets the functoriality correct in all the variables.  \smallskip

For any map of $G$-sets $f\colon S_1\ra S_2$ write $Rf$ for the map in $\cB_G(S_1,S_2)$ corresponding to the span
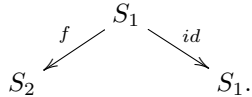
\[\xymatrixrowsep{1pc}\xymatrix{
& S_1\ar[dl]_f  \ar[dr]^{id}\\
S_2 && S_1.
}
\]
Write $If$ for $D(Rf)=[S_1\overset{id}{\leftarrow}S_1\overset{f}{\to} S_2]$.  The following elementary result is very useful:

\begin{prop}
\label{pr:I=R}
If $f\colon S_1\ra S_2$ is an isomorphism of $G$-sets then $Rf=I(f^{-1})$ and $If=R(f^{-1})$ as maps in $\cB_G$.
\end{prop}

\begin{proof}
The first comes from the isomorphism of spans
\[ \xymatrixrowsep{0.6pc}\xymatrix{
& S_1\ar[dl]_-f \ar[dr]^-\id \ar[dd]^f \\
S_2 && S_1. \\
& S_2\ar[ul]^{\id}\ar[ur]_-{f^{-1}}}
\]
For the second, just take inverses (or rotate the diagram 180 degrees).
\end{proof}

We now define a $G$-Mackey functor.

\begin{defn}
\label{de:Mackey}
    A \mdfn{$G$-Mackey functor} is defined to be an additive functor $M\colon \cB_G^{op}\ra \Ab$.  For a map of $G$-sets $f\colon S_1\ra S_2$ we abbreviate $M(Rf)$ as $f^*$ and $M(If)$ as $f_*$. 
\end{defn}
The maps $f^*$ are called \dfn{pullbacks}, and the maps $f_*$ are called \dfn{pushforwards}.  

Because of additivity, a Mackey functor is completely determined by its values on the orbits $G/H$ and the collection of maps $f^*$ and $f_*$ where $f$ is a map between orbits.  Furthermore, we only need to know $f^*$ and $f_*$ for certain maps that generate all others under composition. It is straightforward to check that any map of orbits can be written as a composition of two types of maps:
\begin{enumerate}
    \item[(i)] quotient maps $p:G/K\to G/H$ where $K\subseteq H$, and
    \item[(ii)] right-multiplication-maps $\rho_g\colon G/K\ra G/g^{-1}Kg$ given by $\rho_g(uK)=uKg=ug\cdot (g^{-1}Kg)$.
    \end{enumerate}
Lastly note that if $g\in G$ then the span $I(\rho_g)$ is equivalent to $R(\rho_{g^{-1}})$ by Proposition~\ref{pr:I=R}, and so $(\rho_g)_*=(\rho_{g^{-1}})^*$. Thus the pushforwards of the right action by $G$ are determined by pullbacks.
\smallskip

Given a Mackey functor $M$ and quotient map $p\colon G/K\to G/H$ we refer to the pullback $p^*\colon M(G/H)\to M(G/K)$ as the \textbf{restriction map} and the pushforward $p_*\colon  M(G/K) \to M(G/H)$ as the \textbf{transfer map}. A Mackey functor can be completely described by giving the values on orbits, the restriction maps, the transfer maps, and the maps $(\rho_g)^*$ for each orbit and each $g\in G$. We will often refer to the maps $(\rho_g)^*$ as the \textbf{action maps}.

When we need to reference a particular quotient map $p$ we will write $p_H^K$ or $p_{G/K\to G/H}$.  When the group $G$ gets large it will be handy to use the symbol $\pi$ for projection maps whose target is the trivial $G$-orbit $G/G$.  Also, when $G$ is abelian we will often abbreviate $\rho_g$ to just $g$, thereby writing $g^*=(\rho_g)^*$ and so forth.

\begin{example}
    Suppose $\ell$ is a prime and $C_\ell$ is a cyclic group of order $\ell$ (we use $\ell$ instead of $p$ for primes to avoid confusion with the quotient maps). Recall that $t$ always denotes a chosen generator for $C_\ell$.  A $C_\ell$-Mackey functor $M$ is the data of two abelian groups $M(C_\ell/C_\ell)$ and $M(C_\ell/e)$ as well as three homomorphisms $p_*\colon M(C_\ell/e)\to M(C_\ell/C_\ell)$, $p^*\colon M(C_\ell/C_\ell)\to M(C_\ell/e)$, and $t^*\colon M(C_\ell/e)\to M(C_\ell/e)$ (recall $t_*=(t^*)^{-1}$). These must satisfy $(t^*)^\ell=id$, $t^*\circ p^*=p^*$, $p_*\circ t^*=p_*$, and $p^*\circ p_*=1+t^*+\dots +(t^*)^{\ell-1}$. We often depict this using a so-called Lewis diagram, named from \cite{L2}:
    \[
    \xymatrix{
    M(C_\ell/C_\ell)\ar@/_/[d]_{p^*} \\ M(C_\ell/e)\ar@/_/[u]_{p_*}\ar@(dl,dr)_{t^*}
    }
    \]
    For $C_{\ell^k}$ the groups in the Lewis diagram form a vertical line of length $k+1$; when this takes up an inordinate amount of vertical space, we will change orientation and draw the diagram horizontally rather than vertically.
    We discuss Mackey functors for general cyclic groups in Section~\ref{se:cyclic}. For these the Lewis diagram forms a higher-dimensional rectangular lattice instead of a line, with the number of dimensions equal to the number of prime factors.
\end{example}

We define two important examples of Mackey functors:

\begin{defn}\mbox{}\par
\begin{enumerate}[(1)]
\item 
The \textbf{Burnside functor} $\mA$ is the representable Mackey functor given by $\mA(-)=\cB_G(-,pt)$.
\item The \mdfn{constant coefficient Mackey functor} $\mZ$ is given by $\mZ(G/H)=\Z$ for all orbits $G/H$. Restriction maps and action maps are the identity. One can check that this then forces the transfer maps $p_*\colon \mZ(G/K)\to \mZ(G/H)$ to be multiplication by the index $[H:K]$.
\end{enumerate}
\end{defn}

We will write $1_{G/H}\in \mZ(G/H)=\Z$ for the standard generator.  So then we have $(p_{G/K\ra G/H})^*(1_{G/H})=1_{G/K}$ and $(p_{G/K\ra G/H})_*(1_{G/K})=[H\cln K]\cdot 1_{G/H}$.

There is a map of Mackey functors $\mA\ra \mZ$ defined as follows: for any orbit $G/H$, the map $\mA(G/H)\ra \mZ(G/H)=\Z$ sends a span $[\pt \la T \ra G/H]$ to the number of elements in the fiber of the second map over $eH$. One readily checks that this is compatible with the pullbacks and pushforwards, and this map is a surjection. 

\begin{remark}
A more natural definition of $\mZ$ is that it is the assignment $S\mapsto \Hom_{\Z[G]}(\Z\langle S\rangle,\Z)$.  See the discussion preceding Proposition~\ref{pr:Zmod2} below for how to make this into a functor on $\cB_G^{op}$. 
From this perspective the map $\mA\ra \mZ$ is more transparent. 
\end{remark}

\begin{prop}
\label{pr:I_Z}
The kernel of $\mA\ra \mZ$ is the sub-Mackey functor $I_{\mZ}\subseteq \mA$ generated by the elements $(Rp^K_H\circ Ip^K_H)-[H:K]\cdot \id_{G/H}$ where $K\leq H\leq G$ ranges over all nested pairs of subgroups.  
\end{prop}

\begin{proof} 
It is straightforward to verify that the given elements are in the kernel, and so one gets a surjection $\mA/I_{\mZ}\ra \mZ$.  The group $\mA(G/H)$ is generated by spans $[* \la X \ra G/H]$, and such spans can all be written as sums of spans $[*\la G/K \llra{p} G/H]$ where $K\leq H$.  This span can in turn be written as the composition $[*\la G/H \ra G/H]\circ [G/H\llla{p} G/K \llra{p} G/H]$.  The latter span is $Rp\circ Ip$, and so modulo $I_{\mZ}$ we find that all elements of $\mA(G/H)$ are multiples of $[*\la G/H\ra G/H]$.  That is, $\mA/I_{\mZ}(G/H)$ is cyclic.  Since $\mA/I_{\mZ}\ra \mZ$ is a surjection, it must be an isomorphism.
\end{proof}

\subsection{The box product}
\label{se:box}

Given two $G$-Mackey functors $M$ and $N$, the \textbf{box product} $M\bbox N$ is defined to be the Day convolution, i.e. the left Kan extension in the following diagram:
\begin{equation}
\label{eq:box-defn}
\xymatrix{
\cB_G^{op} \times \cB_G^{op} \ar[d]_{M\times N} \ar[r]^-{\times} & \cB_G^{op} \ar@{.>}@/^2ex/[ddl]^{M\bbox N} \\
\Ab \times \Ab \ar[d]_{\tens} \\
\Ab.
}
\end{equation}

One can check $\und{A}$ is the unit for this product, and that the category of $G$-Mackey functors is a closed symmetric monoidal category (see \cite{B1} for details).  Consequently, the box product is right exact.  
The Mackey functor $\mZ$ is a ring object in this category (often called Green functors in the literature). We will be focused on $\mZ$-modules for the remainder of this paper.

If $R$ is a commutative Mackey ring and $M$ and $N$ are $R$-modules, one can define $M\bbox_R N$ in terms of the usual coequalizer diagram.  As in ordinary ring theory, for a quotient ring $R/I$, if $M$ and $N$ are $R/I$-modules then the canonical quotient map $M\bbox_R N \ra M\bbox_{R/I} N$ is an isomorphism (this follows from $M\bbox_{R/I} N =M\bbox_{R/I} R/I \bbox_R N = M\bbox_R N$).   In particular, applying this to $\mA\ra \mZ$ shows that $\bbox_{\mZ} = \bbox$.  

\subsection{\mdfn{$\mZ$}-modules}
A $\mZ$-module is a module over $\mZ$ in the symmetric monoidal category of $G$-Mackey functors. Because $\mA\ra \mZ$ is a surjection, a Mackey functor has at most one $\mZ$-module structure.  In fact, a Mackey functor $M$ is a $\mZ$-module if and only if the composite $I_{\mZ}\bbox M \ra A\bbox M \ra M$ is zero, where $A\bbox M \ra M$ is the unit isomorphism.   Since Proposition~\ref{pr:I_Z} gives generators for $I_{\mZ}$, we have the following useful characterization:

\begin{prop}
\label{pr:Zmod}
    A Mackey functor $M$ is a $\mZ$-module if and only if for all subgroups $K\leq H \leq G$,  the composite $p_*\circ p^*\colon M(G/K)\to M(G/K)$ is multiplication by the index $[H:K]$. 
\end{prop}

\begin{proof}
Immediate.
\end{proof}

\begin{remark}
The above proposition identifies $\mZ$-modules with what are commonly called \dfn{cohomological Mackey functors} in the algebra literature.  See, for example, \cite[Section 7]{W}. 
The result itself can be found as \cite[Proposition 16.3]{TW} or \cite[Remark 2.13]{Z}.
\end{remark}

An important class of examples of $\mZ$-modules is given by ``fixed point Mackey functors'':

\begin{defn}
For a $\Z[G]$-module $W$ define the \mdfn{fixed point Mackey functor} $\FP(W)$ to be the $G$-Mackey functor defined by $\FP(W)(G/H)=W^H$, with restriction maps the inclusions, action maps given by left multiplication, and transfer maps the evident ``norm'' maps: for $K\subseteq H$ the transfer $W^K\ra W^H$ is $x\mapsto \sum_{hK\in H/K} hx$.
\end{defn} 

Observe in $\FP(W)$ we have that $p_*\circ p^*$ is multiplication by the index of the corresponding subgroups, and so $\FP(W)$ is a $\mZ$-module by Proposition~\ref{pr:Zmod}. For another perspective on $\FP(W)$, see Remark~\ref{re:FP2} below.

The above construction gives a functor $\FP\colon \Z[G]\MMod \to \mZ\MMod$. We also have an evaluation functor $\ev_G\colon \mZ\MMod\to \Z[G]\MMod$ that sends $M$ to $M(G/e)$. This can be checked to give an adjoint pair
\[ \ev_G\colon \mZ\MMod \adjoint \Z[G]\MMod\colon \FP
\]
with $\ev_G$ the left adjoint (see Remark~\ref{re:FP2} below for another perspective on this).\smallskip

Our next characterization utilizes a modified version of the Burnside category. Define $\cBZ_G$ to be the category whose objects are finite $G$-sets, and whose morphisms are given by $\cBZ_G(S,T)=\Hom_{\Z[G]}(\Z\langle S\rangle, \Z\langle T\rangle)$, with the evident composition maps. We have a functor $Q:\cB_G\to \cB\Z_G$ that is the identity on objects, and the map $\cB_G(S,T)\ra \cBZ_G(S,T)$ sends a span $[T\llla{f} X \llra{g} S]$ to the composition
\[ \Z\langle S\rangle \llra{g^*} \Z\langle X\rangle \llra{f_*} \Z\langle T\rangle
\]
where $g^*(s)=\sum_{x\in g^{-1}(s)} x$ for $s\in S$ and $f_*(x)=f(x)$ for $x\in X$.  The following is proven in \cite[Theorem 4.3]{Y}, and also in \cite[Proposition 2.15]{Z}:

\begin{prop}
\label{pr:Zmod2}
    A Mackey functor $M$ is a $\mZ$-module if and only if $M$ factors through $\cB\Z_G$. That is, $M$ is a $\mZ$-module if and only if there is an additive functor $M':\cB\Z_G^{op}\to \Ab$ such that $M=M'\circ Q$. Hence the category of $\mZ$-modules is equivalent to the category of additive functors $\cB\Z_G^{op}\to \Ab$.
\end{prop}

Combining the above result with Proposition~\ref{pr:Zmod} shows that $\cB\Z_G$ may be regarded as the quotient category obtained from $\cB_G$ by imposing the relations $Rp_{(G/K\ra G/H)}\circ Ip_{(G/K\ra G/H)}=[H\colon K]\cdot \id_{G/H}$ whenever $K\leq H\leq G$.  

\begin{remark}
\label{re:FP2}
Let $W$ be a $\Z[G]$-module.  The functor $\Hom_{\Z[G]}(\blank,W)$ is an additive functor $\cB\Z_G^{op} \ra \Ab$, and precomposing this with $Q$ is precisely $\FP(W)$.  
For a $G$-set $S$ there is an evident canonical isomorphism $\Z\langle S\rangle \ra \cB\Z_G(G,S)$, so that if $M$ is a $\mZ$-module we have natural maps
\[ \Z\langle S\rangle \tens M(S) \lra B\Z_G(G,S)\tens M(S) \lra M(G)
\]
which are equivariant with respect to the left $G$-action.
Taking adjoints gives $M(S)\ra \Hom_{\Z[G]}(\Z\langle S\rangle,M(G))=(\FP M(G))(S)$.  This is the map $M\ra \FP M(G)$ forming the unit of the $(\ev_G,\FP)$ adjunction.
\end{remark}

\subsection{The box product of $\mZ$-modules}
Let $M$ and $N$ be two $\mZ$-modules, i.e. functors $\cBZ_G^{op}\ra \Ab$. 
The box product $M\bbox_{\mZ} N$ is defined to be the Day convolution, i.e. the left Kan extension in the diagram analogous to (\ref{eq:box-defn}) where one changes $\cB_G$ to $\cB\Z_G$.  
One can verify that this construction gives the same box product as $M\bbox N$ (see the end of Section~\ref{se:box}), so we drop the $\mZ$ from the box notation without danger of confusion.
In terms of explicit formulas, this means that for any finite $G$-set $X$ the abelian group $(M\bbox N)(X)$ is the coequalizer  of
\begin{center}
\begin{tikzcd}
 \!\!\!\!\!\! \coprod\limits_{C_1,C_2,D_1,D_2}
\cBZ_G(X,C_1\times C_2) \tens \cBZ_G(C_1,D_1)\tens \cBZ_G(C_2,D_2) \tens M(D_1)\tens N(D_2) \arrow[d, shift left] \arrow[d,shift right] 
\\
\coprod\limits_{C_1,C_2} 
\cBZ(X,C_1\times C_2) \tens M(C_1)\tens N(C_2)
\end{tikzcd}
\end{center}
\noindent
where the left arrow just uses composition (and Cartesian product) in $\cBZ_G$ and the right arrow only uses the functor structure on $M$ and $N$. 

Note that $(M\bbox N)(S)$ is not typically $M(S)\tens N(S)$, though there is a natural map from the former to the latter.
To see this let $\Delta_S\in \cBZ_G(S,S\times S)$ be the diagonal, i.e. the span $[S\times S \llla{\Delta} S\llra{=}S]$.  We have the natural maps 
\[ M(S)\tens N(S)\ra \cBZ_G(S,S\times S)\tens M(S)\tens N(S) \ra (M\bbox N)(S)\]
where the first is given by $u\mapsto \Delta_S \tens u$, and the second is from the defining coequalizer for $(M\bbox N)(S)$.  While this comparison is not usually an isomorphism, it is in one important case:

\begin{prop}
\label{pr:box-on-G}
The natural map $M(G)\tens N(G)\ra (M\bbox N)(G)$ is an isomorphism.
\end{prop}

We defer the proof until near the end of Section~\ref{se:res-ind}, when we have enough tools to make it very simple.

Recall the adjoint pair $\ev_G\colon \mZ\MMod \adjoint \Z[G]\MMod\colon FP$, with $\ev_G$ the left adjoint.  The category $\mZ\MMod$ has the box product, and $\Z[G]\MMod$ has the tensor product (over $\Z$) with the diagonal action of $G$.  
Proposition~\ref{pr:box-on-G} shows that $\ev_G$ is strong monoidal, therefore the right adjoint $FP$ is also monoidal; that is, there are natural maps 
\begin{equation}
\label{eq:FP-monoidal}
(\FP M)\bbox (\FP N)\ra \FP(M\tens N).
\end{equation}
Note that $FP$ need not be strong monoidal, as the following example shows.  But it follows from Corollary~\ref{co:F-box-F} below that when $M$ and $N$ are permutation modules (that is, $G$-modules of the form $\Z\langle S\rangle$ where $S$ is a $G$-set) the map of (\ref{eq:FP-monoidal}) is an isomorphism.

\begin{example}
Let $G=C_2$ and let $M=N=\Z_-$ ($\Z$ with the sign action).  Then $M\tens N\iso \Z$ and so $FP(M\tens N)\iso \mZ$.  But $FP(M)=FP(N)$ is the Mackey functor
\[ \xymatrix{
\Z \ar@(ul,ur)[]^{-1} \ar@<0.5ex>[r] & 0 \ar@<0.5ex>[l]
}
\]
and $FP(M)\bbox FP(N)$ is readily checked to be 
\[ \xymatrix{
\Z \ar@(ul,ur)[]^{1} \ar@<0.5ex>[r]^{1} & \Z \ar@<0.5ex>[l]^2
}
\] (note that this is not equal to $\mZ$, as the $p_*$ and $p^*$ are swapped.) The map $FP(M)\bbox FP(N)\ra FP(M\tens N)$ is the identity on the $C_2$-spot and multiplication by $2$ on the $C_2/C_2$-spot. 
\end{example}

\subsection{Free modules}
We next give some constructions of free $\mZ$-modules.

\begin{defn}
    For any finite $G$-set $S$, let $F_S$ be the $\mZ$-module $X\mapsto \cBZ_G(X,S)$.  This is the ``free $\mZ$-module on one generator at spot $S$''. 
\end{defn} 
It is useful to denote $\id\in \cBZ_G(S,S)=F_S(S)$ as $g_S$, as this is the aforementioned generator.  The free module $F_{\pt}$ is $\mZ$. 
By the Yoneda Lemma we have the natural isomorphism
\[  \Hom_{\mZ}(F_S,M) \iso M(S),
\]
which sends a map $\alpha\colon F_S\ra M$ to the element $\alpha(g_S)\in M(S)$.  

\begin{example}
\label{ex:free}
One can engineer free modules in terms of the universal property: loosely speaking, to make $F_{S}$ one drops a generator at spot $S$ and lets it freely generate other elements under pushforwards and pullbacks.  For example, let $\ell$ be a prime and consider $G=C_{\ell^2}$.  The three basic orbits are $\Theta_{\ell^2}$, $\Theta_\ell$, and $\Theta_1$.  Below is a Lewis diagram for $F_{\Theta_\ell}$:
\[ \xymatrixrowsep{0.3pc}\xymatrix{
\Z\langle p^*g,t^*(p^*g),\ldots,(t^*)^{\ell-1}(p^*)g\rangle \ar@<0.5ex>[r]^-{p_*} &
\Z\langle g,t^*g,\ldots,(t^*)^{\ell-1}g\rangle \ar@<0.5ex>[r]^-{p_*} \ar@<0.5ex>[l]^-{p^*}& \Z\langle p_*g\rangle.  \ar@<0.5ex>[l]^-{p^*}\\
\Theta_{\ell^2} & \Theta_\ell & \Theta_1
}
\]
Note the Mackey relations $t^*(p^*g)=p^*(t^*g)$ and $(t^*)^\ell g=g$, from which it follows that $(t^*)^\ell(p^*g)=p^*g$, and so forth.  A good exercise is to work out the action of $p^*$, $p_*$, and $t^*$ on all of the elements in the picture: they all follow from the Mackey relations and the fact that $g$ is an element at spot $\Theta_\ell$.  
\end{example}

The construction method in Example~\ref{ex:free} is useful, but in complicated examples one might wonder how to systematically account for all of the images of the generator $g$.  It turns out that there
are two other ways to describe free modules---one using fixed point Mackey functors and the other using an exponential construction---and in some aspects these are a bit more concrete.  These different descriptions turn out to be surprisingly useful; often a line of argument works very well with one but not the others.  

First we have the following, from \cite[Lemma 16.4]{TW}:

\begin{prop}\label{pr:fixedpointfree}
For any $H\leq G$, the map $F_{G/H}\ra \FP(\Z\langle G/H\rangle)$ sending $g_{G/H}$ to the element $eH\in \Z\langle G/H\rangle^H$ is an isomorphism.
\end{prop}

\begin{proof}
From our development using $\cB\Z_G$ this is almost tautological.  The free module $F_{G/H}$ is the Mackey functor $X\mapsto \cB\Z_G(X,G/H)$, whereas $FP(\Z\langle G/H\rangle)$ is the Mackey functor $X\mapsto \Hom_{\Z[G]}(\Z\langle X\rangle,\Z\langle G/H\rangle)$.
These are the same by the definition of $\cB\Z_G$.  
\end{proof}

We next describe the exponential construction. Let $M$ be a $\mZ$-module and let $S$ be a finite $G$-set.  Define $M^S$ to be the $\mZ$-module $X\mapsto M(S\times X)$.  For a map $f\colon S_1\ra S_2$ in $\cBZ_G$ we obtain $M^{S_2}\ra M^{S_1}$ via the maps $(f\times \id)^*\colon M(S_2\times X)\ra M(S_1\times X)$.  
So $M^{(\blank)}$ is naturally contravariant.  

It is worth noting at this point that $M^S$ and $M^{DS}$ are the same Mackey functor, but the two notations come with different variances in $S$:  a map $f\colon S_1\ra S_2$ in $\cBZ_G$ induces $M^{S_2}\ra M^{S_1}$ and $M^{DS_1}\ra M^{DS_2}$.  The same phenomenon occurs with the free modules $F_S$ and $F_{DS}$, and in many other places in this subject.  As a consequence, certain formulas will include $D$ in order to get the variance correct on the two sides.  

\begin{prop}
There is a natural isomorphism of $\mZ$-modules $\mZ^{DA}\iso F_A$.
\end{prop}

\begin{proof}
Our definitions give $\mZ^{DA}(X)=\mZ(DA\times X)=\cBZ_G(DA\times X,\pt)$ and $F_A(X)=\cBZ_G(X,A)$.  So the result follows from the adjunction isomorphism
\[ \cBZ_G(DA\times X,\pt)\iso \cBZ_G(X,A\times \pt)=\cBZ_G(X,A),
\]
which is natural in both $X$ and $A$ (the latter thanks to the presence of $D$).
\end{proof}

For $\mZ$-modules $M$ and $N$ let $\uHom(M,N)$ be the $\mZ$-module defined by
\[ \uHom(M,N)(X)=\Hom(F_X\bbox M,N).
\]
Then there is an adjunction
\[ \Hom_{\mZ}(M\bbox N,Q) \iso \Hom_{\mZ}(M,\uHom(N,Q)).
\]
The functors $\bbox$ and $\uHom$ give a closed symmetric monoidal structure on the category of $\mZ$-modules.  It is worth stressing the connection between the internal $\Hom$  
and regular $\Hom$, namely
\begin{equation}
\label{eq:uHom=Hom} 
\uHom(M,N)(\pt)=\Hom(M,N).
\end{equation}

The following result connects our exponential construction to both the box and Hom operation with free modules:

\begin{prop}
\label{pr:box-slice-isos}
There is a natural isomorphism $F_A\bbox M\iso M^{DA}$ and also a natural isomorphism $\uHom(F_A,M) \iso M^A$.  Consequently, note the natural isomorphism
\[ \uHom(F_A,\mZ)\iso \mZ^A \iso F_{DA}.
\]
\end{prop}

\begin{remark}
It is worth accentuating that the second statement of the above result shows that free modules are self-dual: $\uHom(F_A,\mZ)\iso F_A$.  However, a map $F_A\ra F_B$ induced by $f\in \cB\Z_G(A,B)$ dualizes to the map $F_B\ra F_A$ induced by $Df$. See also Example~\ref{ex:cyclic-res} below.  
\end{remark}

Proving Proposition~\ref{pr:box-slice-isos} takes a bit of legwork, so we defer the proof to the appendix.  It reduces to the case where $M$ is a free module, and there it can just be checked directly from the definitions.  The result below is basically this case: for expository reasons it is stated here as a corollary, but in fact it is the key step in the proof of the proposition.  

\begin{cor}
\label{co:F-box-F}
There is a natural isomorphism $F_A\bbox F_B\iso F_{A\times B}$.
\end{cor}

\begin{proof}
We have $F_A\bbox F_B\iso (F_B)^{DA}$.  Then we observe for a $G$-set $\Theta$ that
\[ (F_B)^{DA}(\Theta)=F_B(DA\times \Theta)=\cB\Z(DA\times \Theta,B)=\cB\Z(\Theta,A\times B)=F_{A\times B}(\Theta).\]
\end{proof}

Proposition~\ref{pr:box-slice-isos} leads to a couple of important corollaries:

\begin{cor}
\label{co:flat->Free}
The free modules $F_A$ are flat.
\end{cor}

\begin{proof}
Exactness of a sequence is determined by evaluating at $G$-sets $S$, and Proposition~\ref{pr:box-slice-isos} gives natural isomorphisms $(F_A\bbox M)(S)\iso M(A\times S)$.   So if a sequence $M_\bullet$ is exact then so is $F_A\bbox M_\bullet$, for any $A$.   
\end{proof}

\begin{cor}
\label{co:box-hom}
For finite $G$-sets $A$ and $\mZ$-modules $M$ there is a natural isomorphism
$\uHom(F_A,M)\iso \uHom(F_A,\mZ)\bbox M$.
\end{cor}

\begin{proof}
Proposition~\ref{pr:box-slice-isos} gives us
\[ \uHom(F_A,M)\iso M^A \iso F_{DA}\bbox M \iso \uHom(F_A,\mZ)\bbox M.
\]
\end{proof}

\subsection{Maps between free modules}
A map $f\colon S_1\ra S_2$ in $\cBZ_G$ induces the map $F_{S_1}\ra F_{S_2}$ 
sending $g_{S_1}\mapsto f^*(g_{S_2})$, and note that $\cBZ_G(S_1,S_2)\ra \Hom_{\mZ}(F_{S_1},F_{S_2})$ is an isomorphism by the Yoneda Lemma.  Thus a map of free modules $\bigoplus_{i=1}^k F_{S_i} \ra \bigoplus_{j=1}^m F_{T_j}$ is determined by a collection of maps $\alpha_{j,i}\colon S_i\ra T_j$, which can be assembled into the evident $m\times k$ matrix.
Composition then corresponds to matrix multiplication, as usual.
The map sends the generator $g_{S_i}$ to $\sum_{j=1}^l (\alpha_{j,i})^*(g_{T_j})$.

Given $f\colon S_1\ra S_2$, it is tempting to denote the induced map $F_f\colon F_{S_1}\ra F_{S_2}$ as either just ``$f$'' or as ``$f^*$'', in the latter case thinking of the behavior on generators.  Systematic use of the latter convention, though, leads to formulas like $h^*\circ f^*=(hf)^*$, which inherently conflict with the usual contravariance formulas for Mackey functors.  So it is better to avoid this route.  We will tend to abbreviate $F_f$ to just $f$, and remember that on generators the behavior is $g_{S_1}\mapsto f^*g_{S_2}$.  

\begin{example}
\label{ex:cyclic-res}
Fix a prime $\ell$ and a generator $t$ for $C_\ell$.  Let $M$ be the unique $\mZ$-module having $M(C_\ell)=0$ and $M(\pt)=\Z/\ell$.  The element $1\in M(\pt)$ determines a map $\epsilon \colon F_{\pt} \ra M$, which is readily checked to be a surjection.  The following can be verified to a free resolution of $M$:
\[ \xymatrix{
0 \ar[r] & F_{\pt} \ar[r]^{Ip} & F_{C_\ell} \ar[r]^{\id-Rt} & F_{C_{\ell}} \ar[r]^{Rp} & F_{\pt} \ar[r]^\epsilon & M \ar[r] & 0.
}
\]
Here $p\colon C_\ell\ra \pt$ and $t\colon C_\ell\ra C_\ell$ are the maps in the orbit category.  
Observe that $Rp\circ Rt=R(p\circ t)=Rp$ by functoriality.  For $Rt\circ Ip$ recall from Proposition~\ref{pr:I=R} 
that $Rt=It^{-1}=It^{\ell-1}$, and so $Rt\circ Ip=It^{\ell-1}\circ Ip=I(p\circ t^{\ell-1})=Ip$.  This verifies that we have a chain complex, and checking exactness involves just a bit more legwork (this is a good exercise, but details can be found in
Proposition~\ref{pr:Z/Fa-res} below).

Applying $\uHom(\blank,\mZ)$ to the resolution (without the $M$) yields
\[ 
\xymatrix{
0 & F_{\pt} \ar[l] & F_{C_\ell} \ar[l]_{Rp} & F_{C_{\ell}} \ar[l]_{\id-It} & F_{\pt} \ar[l]_{Ip} & 0. \ar[l]
}
\]
Here we have used the isomorphisms $\uHom(F_A,\mZ)\iso F_{DA}$ from Proposition~\ref{pr:box-slice-isos}.  Note the  names of the maps in both complexes and how these follow our conventions.
\end{example}

\begin{remark}
In the future, we will identify any map $f$ in the orbit category with its image $Rf$ in $\cB\Z_G$.  So we will sometimes drop the $R$'s in examples such as the above, though including them can provide some extra clarity.    
\end{remark}

\begin{remark}
\label{re:translate}
Shortly we will be taking free resolutions $F_\bullet$ and applying $(\blank)\bbox M$ or $\uHom(\blank,M)$.  To analyze such constructions, observe that
for a map $f\colon S\ra T$ in $\cB\Z_G$, a $G$-set $\Theta$, and a $\mZ$-module $M$, Proposition~\ref{pr:box-slice-isos} gives natural isomorphisms
\[ \xymatrixcolsep{3.2pc}\xymatrix{
(F_S\bbox M)(\Theta) \ar[r]^{F_f\bbox \id_M}\ar[d]^\iso & (F_T\bbox M)(\Theta)\ar[d]_\iso\\ 
M(S\times \Theta) \ar[r]^-{(Df\times \id)^*} & M(T\times \Theta)
}\ \ 
\xymatrix{
\uHom(F_S,M)(\Theta)\ar[d]^\iso & \uHom(F_T,M)\ar[l]_-{\uHom(F_f,M)}\ar[d]^\iso\\
M(S\times \Theta) & M(T\times \Theta).\ar[l]_{(f\times \id)^*}
}
\]
Note that $(Df\times \id)^*$ can also be written as $(f\times \id)_*$, which is often more convenient.  

In particular, taking the first with $M=\mZ$ recovers the canonical isomorphism
\[ \xymatrixcolsep{3.2pc}\xymatrix{F_S(\Theta)\ar[d]^\iso \ar[r]^{F_f} & F_T(\Theta)\ar[d]_\iso  \\
\mZ(S\times \Theta) \ar[r]^-{(f\times \id)_*} & \mZ(T\times \Theta).
}
\]
\end{remark}

\subsection{Rationalization}
We end this section with a nice result about rationalization: 

\begin{prop}
\label{pr:rational-equiv}
For any $\mZ$-module $M$, the natural map $M\ra \FP(M(G))$ is a rational equivalence.  
\end{prop}

\begin{proof}
Evaluating at the orbit $G/H$, the map in question is $p^*\colon M(G/H)\ra M(G)^H$ where $p\colon G\ra G/H$ is the projection.    
We have the maps
\[ M(G/H) \llra{p^*} M(G) \llra{p_*} M(G/H),
\]
with the image of the first landing in $M(G)^H$.  
The composition $p_*\circ p^*$ is multiplication by $|H|$, and the composition $p^*\circ p_*$ equals $\sum_{h\in H} (\rho_h)^*$---which is multiplication by $|H|$ on $M(G)^H$.  So after rationalization the map $p^*\colon M(G/H)\ra M(G)^H$ is an isomorphism.    
\end{proof}

\begin{cor}
\label{co:rational-equiv1}
The functors $\FP\colon \Q[G]\MMod \adjoint \mQ_G\MMod\colon \ev_G$ are an equivalence.
\end{cor}

\begin{proof}
The composition $\ev_G\circ \FP$ is naturally isomorphic to $\id$ even over $\Z$.  The natural isomorphism $\FP\circ \ev_G\iso \id$ is by Proposition~\ref{pr:rational-equiv}.
\end{proof}

Of course the above results do not require inverting all integers, one only needs to  invert $|G|$.

\section{Families of subgroups and their associated ideals}
\label{se:fam-ideals}

Let $G$ be a finite group.  A \mdfn{$G$-family of subgroups} is a collection $\cF$ of subgroups of $G$ that is closed under taking subgroups and under conjugation.  More briefly, we just say that $\cF$ is closed under subconjugacy.  Given such a family, let $I_\cF\subseteq \mZ$ be the smallest sub-Mackey functor containing the elements $1\in \mZ(G/H)$ for all $H\in \cF$.  This is an ideal in the Mackey ring $\mZ$.  We will often write $\mZ/\cF$ as an abbreviation for the quotient $\mZ/I_\cF$.  

Note that the element $1\in \mZ(G/H)$ determines a map $F_{G/H}\ra \mZ$. 
The Mackey functor $\mZ/\cF$ can also be described as the cokernel in the sequence
\[ \bigoplus_{H\in \cF} F_{G/H} \lra \mZ \lra \mZ/\cF \lra 0
\]
where the left map sends each generator $g_{G/H}$ to $1_{G/H}$.  

\begin{example}\label{ex:c45}
Let $G=C_{45}$, the cyclic group of order $45$. Consider the family $\cF=\{e,C_3\}$. From left to right, the diagrams in Figure~\ref{fig:Z/I-example}
show the orbit category (without the automorphisms drawn), the Mackey functor $\mZ_{C_{45}}$, the ideal $I_{\cF}$, and the quotient $\mZ/\cF$. Note that restriction maps $p^*$ are drawn as straight arrows, and transfers $p_*$ are drawn curved; the $t^*$ maps are not drawn at all, as they are all equal to the identity.  Unlabelled maps send $1\mapsto 1$.  In the orbit category we write $\Theta_{d}$ for the unique orbit of size $d$ (so $\Theta_d=C_{45}/C_{45/d}$). We can imagine placing a $1$ at the spots $\Theta_{45}=C_{45}/e$ and $\Theta_{15}=C_{45}/C_3$ and letting those $1$'s freely generate an ideal.

\begin{figure}[ht]
\label{fig:Z/I-example}
\caption{The $\mZ$-modules $\mZ$, $I_\cF$, and $\mZ/\cF$}
\[
\xymatrix{
\Theta_{1}& \Theta_5 \ar[l] & \Z\ar[r]\ar[d] & \Z \ar@/_2.5ex/[l]_5\ar[d] \\
\Theta_{3} \ar[u]  & \Theta_{15}\ar[u]\ar[l] & \Z\ar[d] \ar@/^2.5ex/[u]^3\ar[r] & \Z \ar@/_2.5ex/[l]_5\ar[d]\ar@/_2.5ex/[u]_3\\
\Theta_{9}\ar[u] & \Theta_{45}\ar[u]\ar[l] & \Z\ar[r] \ar@/^2.5ex/[u]^3& \Z \ar@/_2.5ex/[l]_5\ar@/_2.5ex/[u]_3\\
{} \ar@<1.5ex>@{}[r]^{\Or(C_{45})} &{}& {}\ar@{}@<1.5ex>[r]^{\mZ} & {}
    }\hspace{0.1in}
\xymatrix{
15\Z\ar[r]\ar[d] & 3\Z \ar@/_2.5ex/[l]_5\ar[d] \\
5\Z\ar[d] \ar@/^2.5ex/[u]^3\ar[r] & \Z \ar@/_2.5ex/[l]_5\ar[d]\ar@/_2.5ex/[u]_3\\
5\Z\ar[r] \ar@/^2.5ex/[u]^3& \Z \ar@/_2.5ex/[l]_5\ar@/_2.5ex/[u]_3\\
{}\ar@{}@<0.9ex>[r]^{I_\cF} & {}
}
\hspace{0.1in}
\xymatrix{
\Z/15\ar[r]\ar[d] & \Z/3 \ar@/_2.5ex/[l]_5\ar[d] \\
\Z/5\ar[d] \ar@/^2.5ex/[u]^3\ar[r] & 0 \ar@/_2.5ex/[l]\ar[d]\ar@/_2.5ex/[u]\\
\Z/5\ar[r] \ar@/^2.5ex/[u]^3& 0 \ar@/_2.5ex/[l]\ar@/_2.5ex/[u]\\
{}\ar@{}@<2.5ex>[r]^{\mZ/I_\cF} &{}
}
\]
\end{figure}
\end{example}

\smallskip

In general the value of $\mZ/\cF$ at an orbit $G/J$ is given by the following formula:

\begin{prop}
\label{pr:Z/F-order}
For any subgroup $J$ of $G$ the group $(\mZ/\cF)(G/J)$ is cyclic of order
equal to 
\[ \gcd\bigl \{ [J:H]\,\bigl |\, H\in \cF,\  H\leq J\bigr \} = \frac{\#J}{\lcm\{ \#H\,|\, H\in \cF, H\leq J\} }.    
\]
In particular, the order of $G$ annihilates $\mZ/\cF$.
\end{prop}

\begin{proof}
The group $(\mZ/\cF)(G/J)$ is a quotient of $\mZ(G/J)=\Z$, and so is cyclic.  Let $D_J=\{ [J:H]\, \bigl |\, H\in \cF, H\leq J\}$ and set $d_J=\gcd(D_J)$. 

First note that if $H\in \cF$ and $H\leq J$ then the transfer map $\mZ(G/H) \ra \mZ(G/J)$ sends $1_{G/H}$ to $[J:H]\cdot 1_{G/J}$.   
From this it follows at once that $d_J\cdot 1_{G/J}\in I_{\cF}(G/J)$, and so $(\mZ/\cF)(G/J)$ is annihilated by $d_J$.  Proving that $d_J$ is the precise order requires some additional work.

For any $K\in \cF$ we will analyze the map $f\colon F_{G/K}\ra \mZ$ that sends the generator to $1_{G/K}$.  We will use that $F_{G/K}=\FP(\Z\langle G/K\rangle)$, with generator $g_{G/K}$ corresponding to the element $[eK]\in \Z\langle G/K\rangle^K$ (see Proposition~\ref{pr:fixedpointfree}).  The map $f_{G/e}\colon \Z\langle G/K\rangle \ra \Z$ (the component of $f$ at the $G/e$ orbit) sends $[eK]$ to $1_G$, and so by $G$-equivariance it must send $[gK]$ to $1_G$ for all $g\in G$.  

Next consider a subgroup $J\leq G$. We can determine $f_{G/J}$ using the diagram of restriction maps
\[
\xymatrix{
F_{G/K}(G/J)\ar@{=}[r]\ar[d]_{p^*} & \Z\langle G/K\rangle^J \ar@{^{(}->}[d]_{p^*} \ar[r]^-{f_{G/J}} & \Z \ar[d]_{1}\\
F_{G/K}(G/e) \ar@{=}[r] & \Z\langle G/K\rangle \ar[r]^-{f_{G/e}} & \Z. 
}
\]
The restriction map in the middle is just the inclusion of fixed points.  Write $G/K=S_1\amalg S_2\amalg \cdots \amalg S_r$ as $J$-sets, where each $S_j$ is a $J$-orbit.  Then $\Z\langle G/K\rangle^J$ is free abelian with generators the sums of elements in each $S_j$.  Recall $f_{G/e}$ maps each element of $G/K$ to $1$, so under the composition $f_{G/e}\circ p^*$ these sums map to $\# S_j\cdot 1_G$. Thus under $f_{G/J}$ they map to $\# S_j\cdot 1_{G/J}$.  This shows that the image of $f_{G/J}$ is generated by the gcd of the numbers $\# S_j$.  Finally note if $gK$ is an element of $S_j$ then $S_j\iso J/(J\cap gKg^{-1})$ as $J$-sets.  

Taking all possible $g\in G$ and $K\in \cF$, we find that $I_\cF(G/J)$ is generated by $m\cdot 1_{G/J}$ where $m$ is the gcd of the set
\[ \bigl \{ [J\colon J\cap gKg^{-1}] \,\bigl |\, K\in \cF, \ g\in G\bigr \}.
\]
But since $\cF$ is closed under subconjugacy, this is precisely the set $D_J$. So $(\mZ/\cF)(G/J)$ has order $d_J$.
\end{proof}

The above result leads to the following `recognition theorem' for the $\mZ$-modules $\mZ/\cF$.  A more streamlined result along these lines that works only for cyclic $G$ is given in Proposition~\ref{pr:Z/F_a-identify}.

\begin{prop}
\label{pr:recognition-easy}
Let $\cF$ be a family of subgroups of $G$.  
Let $M$ be a $\mZ$-module having the properties that
\begin{enumerate}[(1)]
\item $M(\pt)$ is a cyclic abelian group;
\item all restriction maps are surjections; and
\item the order of each $M(G/J)$ is equal to the number from Proposition~\ref{pr:Z/F-order}.  
\end{enumerate}
Then $M\iso \mZ/\cF$.  Moreover, for any generator $a\in M(\pt)$ the induced map $\mZ\ra M$ sending $1_{\pt}\mapsto a$ factors through $\mZ/\cF$ to give an isomorphism $\mZ/\cF\ra M$.  
\end{prop}

\begin{proof}
Let $a$ be any generator for $M(\pt)$ and consider the associated map $f\colon \mZ\ra M$ defined by $1_{pt}\mapsto a$.  
Conditions (1) and (2) imply that $f$ is a surjection.  For $H\in \cF$ the order of $M(G/H)$ is $0$ by condition (3).  So $f$ annihilates $I_\cF$, and hence factors as a (necessarily surjective) map $\mZ/\cF\ra M$.  
At each spot $G/J$ this is a surjection between cyclic groups of the same order so it is an isomorphism.  
\end{proof}

Note that if $\cF_1$ and $\cF_2$ are two $G$-families then so are $\cF_1\cup \cF_2$ and $\cF_1\cap \cF_2$. The next proposition details how these operations on families relate to operations on the corresponding ideals or quotients.

\begin{prop}\label{pr:familyprop}
Let $\cF_1$ and $\cF_2$ be two $G$-families.  Then
\begin{enumerate}[(a)]
\item If $\cF_1\subseteq \cF_2$ then $I_{\cF_1}\subseteq I_{\cF_2}$.  
\item $I_{\cF_1\cup \cF_2}=I_{\cF_1}+I_{\cF_2}$.
\item $I_{\cF_1\cap \cF_2}\subseteq I_{\cF_1}\cap I_{\cF_2}$.
\item $\mZ/\cF_1 \,\bbox\, \mZ/\cF_2\iso \mZ/(\cF_1\cup \cF_2)$.  
\end{enumerate}
\end{prop}

\begin{proof}
Parts (a)--(c) are routine.  For (d) take the box product of the two sequences
\[ \bigoplus_{H\in \cF_1} F_{G/H}\ra \mZ \ra \mZ/\cF_1 \ra 0 \quad\text{and}\quad
\bigoplus_{J\in \cF_2} F_{G/J}\ra \mZ \ra \mZ/\cF_2 \ra 0.
\]
Right-exactness of the box product then gives the exact sequence
\[ \bigoplus_{H\in \cF_1} F_{G/H} \oplus \bigoplus_{J\in \cF_2} F_{G/J} \lra \mZ \lra \mZ/\cF_1\bbox\, \mZ/\cF_2\ra 0
\]
and the result follows since the cokernel of the first map is evidently $\mZ/(\cF_1\cup \cF_2)$.  
\end{proof}

\begin{remark}
The subset in part (c) need not be an equality.  For example, take $G=C_2\times C_2$.  Let $\cF_1=\{e,C_2\times e\}$ and $\cF_2=\{e,e\times C_2\}$.  Then $I_{\cF_1\cap \cF_2}(G)=(4)$ whereas $I_{\cF_1}(G) \cap I_{\cF_2}(G)=(2)$.  
\end{remark}

We now give some definitions that are useful for describing operations involving the modules $I_{\cF}$ and $\mZ/\cF$.

\begin{defn}
Let $\cF$ be a family of subgroups and let $M$ be a $\mZ$-module.
\begin{enumerate}[(1)]
\item $M$ is said to be \mdfn{$\cF$-null} if $M(G/H)=0$ for all $H\in \cF$.
\item $M\{\cF\}$ is defined to be the intersection of all submodules $N\subseteq M$ such that $N(G/H)=M(G/H)$ for all $H\in \cF$.
\item $\Ann_\cF M$ is the submodule of $M$ consisting of all elements $x\in M(G/K)$ such that for every $f\colon G/H\ra G/K$ with $H\in \cF$ one has $f^*(x)=0$.   
\end{enumerate}
\end{defn}

Checking that $\Ann_\cF M$ is closed under pullbacks is immediate. Then to see it is closed under pushforwards (so that it is indeed a submodule as claimed), one observes the following:

\begin{lemma}
\label{le:product-family}
Suppose $S\ra T$ and $G/H\ra T$ are maps of $G$-sets.  If $H\in \cF$ then the pullback $G/H \times_T S$ is a disjoint union of $G/K_i$ where each $K_i\in \cF$.
\end{lemma}

\begin{proof}
The stabilizer of a point $(aH,s)\in G/H\times_T S$ is $aHa^{-1}\cap \stab_G s$, which is subconjugate to $H$ and therefore belongs to $\cF$.  
\end{proof}

We also use this lemma to prove the following vanishing properties for boxing with $\cF$-null modules:

\begin{prop}
\label{pr:F-null-1}
Let $\cF$ be a family of subgroups and
let $M$ be a $\mZ$-module that is $\cF$-null.  Then $F_{G/H}\bbox M=0$ for every $H\in \cF$, and also $I_\cF\bbox M=0$.
\end{prop}

\begin{proof}
We know that $F_{G/H}\bbox M\iso M^{D(G/H)}$ by Proposition~\ref{pr:box-slice-isos}, and hence $[F_{G/H}\bbox M](\Theta)\iso M(G/H\times \Theta)$ for every $G$-set $\Theta$.  If $H\in \cF$ then by Lemma~\ref{le:product-family}  the $G$-set $G/H\times \Theta$ is a disjoint union of $G$-sets $G/K$ with $K\in \cF$. But $M$ vanishes at each of those spots because it was assumed to be $\cF$-null, and so $F_{G/H}\bbox M=0$.

For the second part we have a surjection $\bigoplus_{H\in \cF} F_{G/H} \fib I_\cF$ because $I_\cF$ is generated by the elements $1_{G/H}$ for $H\in \cF$.
Boxing with $M$ gives an induced surjection
\[ \bigoplus_{H\in \cF} F_{G/H}\bbox M \fib I_\cF \bbox M.
\]
Since all of the $F_{G/H}\bbox M$ summands vanish, so does $I_\cF\bbox M$.
\end{proof}

As in classical ring theory, $\mZ/\cF$-modules may be identified with $\mZ$-modules that are annihilated by the ideal $I_\cF$.  The following proposition unwinds this a bit:

\begin{prop}
\label{pr:Z/F-modules}
A $\mZ$-module $M$ is a $\mZ/\cF$-module if and only if $M(G/H)=0$ for all $H\in \cF$ (i.e. $M$ is $\cF$-null).
\end{prop}

To prove this we start with a lemma:

\begin{lemma}
$\mZ/\cF\,\bbox M$ is $\cF$-null for any $\mZ$-module $M$.
\end{lemma}

\begin{proof}
Recall that we have a right-exact sequence $\bigoplus_{H\in \cF} F_{G/H} \ra \mZ \ra \mZ/\cF\ra 0$.  Boxing with $M$ and evaluating at a $G$-set $\Theta$ gives
\[ \xymatrix{
\bigoplus_{H\in \cF} M(G/H\times \Theta) \ar[r]^-{(\pi\times \id)_*} & M(\pt \times \Theta) \ar[r] & (\mZ/\cF\,\bbox M)(\Theta) \ar[r] & 0.
}
\]
(Here we are using the identifications explained in Remark~\ref{re:translate}.) Now suppose that $J\in \cF$ and take $\Theta=G/J$.  Then $(\pi \times \id)_*\colon M(G/J\times G/J)\ra M(\pt \times G/J)$ is surjective because of the identity $(\pi\times \id)_*\circ \Delta_*=\id$.  It follows that $(\mZ/\cF\,\bbox M)(G/J)=0$, as desired.
\end{proof}

\begin{proof}[Proof of Proposition~\ref{pr:Z/F-modules}]
The lemma shows that $\mZ/\cF\,\bbox M$ is $\cF$-null for every $\mZ$-module $M$.
If $M$ is a $\mZ/\cF$-module then it is a retract of $\mZ/\cF\bbox M$ because $M=\mZ\bbox M\ra \mZ/\cF\,\bbox M\ra M$ is the identity.  Since $\mZ/\cF\,\bbox M$ is $\cF$-null, so is $M$.

Conversely, assume that a $\mZ$-module $M$ is $\cF$-null.  Then $I_\cF\bbox M=0$ by Proposition~\ref{pr:F-null-1}, and so $I_\cF$ annihilates $M$.  Therefore $M$ is a $\mZ/\cF$-module.
\end{proof}

We use this characterization of $\mZ/\cF$-modules to prove two final propositions that will be helpful in later sections.

\begin{prop}
\label{pr:M{F}}
For $M$ a $\mZ$-module there is a natural isomorphism $\mZ/\cF\,\bbox M \iso M/M\{\cF\}$ making the diagram
\[ \xymatrix{
\mZ\bbox M \ar@{=}[r]\ar@{->>}[d] & M\ar@{->>}[d] \\
\mZ/\cF\,\bbox M \ar[r]^\iso & M/M\{\cF\}
}
\]
commute.  In particular, the image of $I_\cF\bbox M \ra \mZ\bbox M=M$ is precisely $M\{\cF\}$.  
\end{prop}

\begin{proof}
We use the identification of $\mZ/\cF$-modules with the $\mZ$-modules that are $\cF$-null.  The forgetful/inclusion functor to $\mZ$-modules has left adjoint $M\mapsto \mZ/\cF\,\bbox M$ (looked at the first way), and left adjoint $M\mapsto M/M\{\cF\}$ (looked at the second way).  So these two functors are naturally isomorphic.  
\end{proof} 

\begin{prop}\mbox{}\par
\label{pr:Hom-Ann}
\begin{enumerate}[(a)]
\item A $\mZ$-module $M$ is $\cF$-null if and only if $\Ann_{\cF}M=M$. 
\item If $M$ is a $\mZ$-module and $N\subseteq M$ is a submodule that is $\cF$-null, then $N\subseteq \Ann_{\cF}M$.  
\item $\Ann_\cF M$ is the largest submodule of $M$ that is $\cF$-null.
\item If $M$ is a $\mZ$-module then
the map $\uHom(\mZ/\cF,M) \ra \uHom(\mZ,M)=M$ induced by $\mZ\ra \mZ/\cF$ is an injection with image equal to $\Ann_\cF M$.  
\end{enumerate}
\end{prop}

\begin{proof}
Parts (a)--(c) are simple exercises.  For (d),
injectivity follows from the fact that $\mZ\ra \mZ/\cF$ is surjective.  
The module $\uHom(\mZ/\cF,M)$ is a $\mZ/\cF$-module and hence $\cF$-null, so the image lies inside $\Ann_{\cF}M$.  One can conclude by observing that the inclusion functor from $\cF$-null modules to $\mZ$-modules has both $\Ann_{\cF}(\blank)$ and $\uHom(\mZ/\cF,\blank)$ as right adjoints.  So the functors are equal.
[Note: There is also a direct proof where one applies $\uHom(\blank,M)$ to the right-exact sequence $\bigoplus_{H\in \cF} F_{G/H} \ra \mZ \ra \mZ/\cF\ra 0$ and analyzes what is happening there.]
\end{proof}

\subsection{Connections with Bredon cohomology}
\label{se:Bredon}
The modules $\mZ/\cF$ are of particular interest because they arise naturally in the context of $RO(G)$-graded Bredon cohomology.  
Let $V$ be a nonzero $G$-representation and define the set $\cF_{V}=\{H\leq G \mid V^H\neq 0\}$. It is straightforward to check this set forms a family of subgroups. We have the following result:

\begin{prop}
Let $V$ be a nonzero $G$-representation. Then $\underline{H}^V(pt;\mZ)\cong \mZ/\cF_V$.
\end{prop}
\begin{proof}
This is just a reinterpretation of \cite[Proposition 2.7]{DH1}.
Start with the isomorphism $\uH^V(\pt)\iso \underline{\pi}_0(AG(S^V))$ where $AG$ denotes the free abelian group on the pointed space $(S^V,\infty)$, where the basepoint $[\infty]$ is the zero element.  For brevity write $M$ for $\underline{\pi}_0(AG(S^V))$.  Then
\[ M(G/J)=[G/J_+,AG(S^V)]_G=[S^0,AG(S^V)]_J
\]
where the latter indicates homotopy classes of pointed $J$-equivariant maps and therefore only depends on $V$ as a $J$-representation.  The result \cite[Proposition 2.7]{DH1} shows that the group on the right is cyclic, generated by the map $a$ sending the non-basepoint to $0$, and has order equal to $\gcd \{ [J:H]\,\bigl |\, H\leq J,\ V^H\neq 0\}$.  It follows at once that $M$ satisfies the conditions of Proposition~\ref{pr:recognition-easy}, and therefore $M\iso \mZ/\cF_V$.  
\end{proof}

For more connections with Bredon cohomology see Section~\ref{se:appl} below.


\section{Homological computations for cyclic groups}
\label{se:cyclic}

In this section we specialize to $G=C_n$, the cyclic group of order $n$.
For any positive divisor $d|n$ let $\cF_d$ be the family of subgroups of $G$ whose order divides $\frac{n}{d}$, which is also the family of subgroups of $C_{\frac{n}{d}}$.  Write $I_d=I_{\cF_d}$. We study some general properties of the modules $\mZ/I_d$, compute their free resolutions, and then proceed to calculate $\uTor_*(\mZ/I_d,M)$ and $\uExt^*(\mZ/I_d,M)$ for any $\mZ$-module $M$.\smallskip

Recall that $\Theta_d=C_n/C_{\frac{n}{d}}$, and if $a|b$ then $p^{b}_a\colon \Theta_b\ra \Theta_a$ is the canonical projection. 
Note that $I_d\subseteq \mZ$ is the ideal generated by the element $1_d\in \mZ(\Theta_d)$  (equivalently, by $1_a\in \mZ(\Theta_a)$ for all $d| a$).  
For example the family $\cF=\{e,C_3\}$ of $G=C_{45}$ given in Example~\ref{ex:c45} is $\cF_{15}$, and the ideal $I_{15}$ is generated by the identity element at the orbit $\Theta_{15}$.

If $d|e$ then note that $\cF_e\subseteq \cF_d$ and so $I_{e}\subseteq I_{d}$. We also have the following relations:  

\begin{prop}\label{pr:cyclicfam} If $d|n$ and $e|n$ then
$\cF_d\cap \cF_e=\cF_{[d,e]}$ and $I_d+I_e=I_{(d,e)}$.  
\end{prop}

\begin{proof}
The first statement follows from the definitions and that $(\frac{n}{d},\frac{n}{e})=\frac{n}{[d,e]}$.  

For the second statement, we know $I_d\subseteq I_{(d,e)}$ and $I_e\subseteq I_{(d,e)}$, therefore $I_d+I_e\subseteq I_{(d,e)}$. It suffices to show $1_{(d,e)}\in (I_d+I_e)(\Theta_{(d,e)})$ for the other containment. Observe in the $\Theta_{(d,e)}$-spot the ideal $I_d$ contains the element $[p^d_{(d,e)}]_*(1_d)=\frac{d}{(d,e)}\cdot 1_{(d,e)}$.  Likewise, the ideal $I_e$ contains the element
$[p^e_{(d,e)}]_*(1_e)=\frac{e}{(d,e)}\cdot 1_{(d,e)}$.  So both these elements are in $I_d+I_e$.  Since $\frac{d}{(d,e)}$ and $\frac{e}{(d,e)}$ are relatively prime, it follows that $1_{(d,e)}$ is in $I_d+I_e$.  Hence $I_{(d,e)}\subseteq I_d+I_e$ and we have proven equality.  
\end{proof}

\begin{remark}
It is not true that every family of subgroups $\cF$ of $C_n$ has the form $\cF_d$ for some $d$, but it is true that $I_{\cF}=I_{d}$ for some $d$.  To see this, observe that one can always write $\cF=\bigcup_i \cF_{a_i}$ for some divisors $a_i$, and therefore $I_{\cF}=\Sigma_{i} I_{\cF_{a_i}}=I_{(a_1,\ldots,a_s)}$ by Propositions~\ref{pr:familyprop}(b) and  \ref{pr:cyclicfam}.
\end{remark}

We also note the following immediate consequence of  Propositions~\ref{pr:familyprop} and \ref{pr:cyclicfam} that gives an analog of the familiar fact that $\Z/d\otimes \Z/e\cong \Z/(d,e)$.
\begin{cor}
Let $d,e$ be divisors of $n$. Then $\mZ/I_d\bbox\, \mZ/I_e\cong \mZ/I_{(d,e)}$.
\end{cor}

We next prove some useful properties of the modules $\mZ/I_a$.

\begin{prop}
Fix $e|n$. Let $M$ be a $\mZ$-module that is generated in spots $\Theta_d$ for $e|d$.  Then $M\bbox \mZ/I_e=0$.   In particular, if $e|d$ then $F_d\bbox \mZ/I_e=0$ and $I_d\bbox \mZ/I_e=0$.  
\end{prop}

\begin{proof}
This is by Proposition~\ref{pr:M{F}}: $M\bbox\, \mZ/I_e\iso M/M\{\cF_e\}=M/M=0$.  
\end{proof}

\begin{prop}
\label{pr:Z/F_a-props}
Let $a|n$. The $\mZ$-module $\mZ/I_a$ has the following properties:
\begin{enumerate}[(1)]
\item $[\mZ/I_a](\Theta_d)\iso \Z/\frac{a}{(d,a)}$ for every $d|n$ (in particular, the group is zero when $a|d$).
\item All of the restriction maps are surjections and all of the transfer maps are injections.
\item For every $d|n$, the restriction and transfer maps between the values at $\Theta_d$ and $\Theta_{(d,a)}$ are both isomorphisms.
\end{enumerate}
\end{prop}

\begin{proof}
For (1) observe by Proposition~\ref{pr:Z/F-order} that $[\mZ/I_a](\Theta_d)$ is cyclic and the order is equal to $\gcd\{[C_{n/d}:H]\mid H \in \cF_a, H\leq C_{n/d}\}$. Note that any such $H$ is a cyclic group of order $n/b$ for some divisor $b$ of $n$. We have that $H \in \cF_a$ and $H\leq C_{n/d}$ if and only if $\frac{n}{b} \mid \frac{n}{a}$ and $\frac{n}{b} \mid \frac{n}{d}$, that is, if and only if the divisor $b$ is a common multiple of $a$ and $d$. Thus we have that
\begin{align*}
\gcd\left \{[C_{n/d}:H] \Bigl | H\in \cF_a, H\leq C_{n/d} \right \} =\gcd\left\{ \tfrac{b}{d} \,\Bigl |\  \text{$ [a,d]$ divides $b$} \right\}
&= \frac{[a,d]}{d}=\frac{a}{(a,d)}.
\end{align*}

For (2) note that the restriction maps in $\mZ$ are all surjections, and this property passes to any quotient.
The transfer map $(p_d^e)_*:\Z/\frac{a}{(a,e)}\to \Z/\frac{a}{(a,d)}$ will be multiplication by $\frac{e}{d}$. The order of $\frac{e}{d}$ in $\Z/\frac{a}{(a,d)}$ is computed by dividing $\frac{a}{(a,d)}$ by the gcd of $\frac{e}{d}$ and $\frac{a}{(a,d)}$. Observe
\begin{equation}
\label{eq:exotic}
\left(\frac{a}{(a,d)}, \frac{e}{d}\right) = \left(\frac{a}{(a,d)}, \frac{e}{(a,d)}\right) = \frac{(a,e)}{(a,d)}
\end{equation}
(see Remark~\ref{re:gcd}).  
Thus the order of the element $\frac{e}{d}$ is
\[
\frac{a}{(a,d)} \cdot \frac{(a,d)}{(a,e)} = \frac{a}{(a,e)}.
\]
Since the generator of $\Z/\frac{a}{(a,e)}$ is sent to an element of order $a/(a,e)$, we conclude the transfer map must be an injection.

Finally (3) follows from (2) because $[\mZ/I_a](\Theta_d)\iso \Z/\frac{a}{(a,d)}\iso [\mZ/I_a](\Theta_{(a,d)})$.
\end{proof}

\begin{cor}
\label{co:Z/F_a-injection}
If a map $\mZ/I_a\ra M$ is an injection in the $\Theta_1$ spot, then it is an injection in every spot.
\end{cor}

\begin{proof}
Suppose that $[\mZ/I_a](\Theta_d)\ra M(\Theta_d)$ has a nonzero element in the kernel.  Applying the transfer map to $\Theta_1$ then gives an element in the kernel of $[\mZ/I_a](\Theta_1)\ra M(\Theta_1)$.  But this kernel is zero by hypothesis.  Since the transfer maps in $\mZ/I_a$ are injections, this is a contradiction.
\end{proof}

If $I$ and $J$ are ideals in $\mZ$ then define $I\cdot J$ to be the image of $I\bbox J \ra \mZ\bbox \mZ \llra{\mu} \mZ$.  One has $I\cdot J\subseteq I\cdot \mZ=I$ and likewise $I\cdot J\subseteq J$, so $I\cdot J\subseteq I\cap J$.

\begin{cor}
\label{co:I-product}
Let $a|n$ and $b|n$.  Then $I_a\cap I_b=I_a\cdot I_b=I_{[a,b]}$.  Moreover, the natural map $I_b/I_{[a,b]}\ra I_{(a,b)}/I_a$ is an isomorphism.
\end{cor}

\begin{proof}
Since $1_{[a,b]}\in I_a(\Theta_{[a,b]})$ and $1_{[a,b]}\in I_b(\Theta_{[a,b]})$, the product $1_{[a,b]}\cdot 1_{[a,b]}\in (I_a\cdot I_b)(\Theta_{[a,b]})$.  It follows that $I_{[a,b]}\subseteq I_a\cdot I_b$.  

Since $I_{[a,b]}\subseteq I_a\cdot I_b\subseteq I_a\cap I_b$, it will suffice to show $I_{[a,b]}=I_a\cap I_b$.  Let $d|n$.  By Proposition~\ref{pr:Z/F_a-props}(1) we know $I_a(\Theta_d)\subseteq \Z$ is the ideal generated by $\frac{a}{(a,d)}$ and $I_b(\Theta_d)\subseteq \Z$ is the ideal generated by $\frac{b}{(b,d)}$.  So $(I_a\cap I_b)(\Theta_d)$ is the ideal generated by their lcm.  But one checks that
\[ \bigl [ \tfrac{a}{(a,d)}, \tfrac{b}{(b,d)} \bigr ] = \tfrac{[a,b]}{([a,b],d)},
\]
and Proposition~\ref{pr:Z/F_a-props} says that the expression on the right is the generator for $I_{[a,b]}(\Theta_d)$. 

The isomorphism $I_b/I_{[a,b]}\ra I_{(a,b)}/I_a$ follows directly from $I_a+I_b=I_{(a,b)}$ and $I_a\cap I_b=I_{[a,b]}$ (some textbooks call this the ``Second Isomorphism Theorem'').
\end{proof}

Proposition~\ref{pr:Z/F_a-props} also leads to the following improvement on
Proposition~\ref{pr:recognition-easy}, giving a useful `recognition theorem' for the modules $\mZ/I_a$:

\begin{prop}
\label{pr:Z/F_a-identify}
Let $a|n$.  
Let $M$ be a $\mZ$-module with the following properties: 
\begin{enumerate}[(i)]
\item $M(\Theta_1)\iso \Z/a$; 
\item $M(\Theta_a)=0$; and 
\item all of the restriction maps in $M$ are surjections.  
\end{enumerate}
Then $M\iso \mZ/I_a$.  Moreover, if $u$ is any generator for the abelian group $M(\Theta_1)$ then there is a unique isomorphism $\mZ/I_a\ra M$ that sends $1_{\Theta_1}$ to $u$.  
\end{prop}

\begin{proof}
Let $u$ be a generator for the abelian group $M(\Theta_1)$.  There is an induced map $\mZ\ra M$ sending $1_{\Theta_1}$ to $u$.  Since $M(\Theta_a)=0$ this map must send $1_a$ to $0$, and hence induces a map $\mZ/I_a\ra M$.  By construction, in the $\Theta_1$ spot this is a surjection $\Z/a\ra M(\Theta_1)$.  Since $M(\Theta_1)\iso \Z/a$, this map is in fact an isomorphism.  Then by Corollary~\ref{co:Z/F_a-injection} the map $\mZ/I_a\ra M$ is an injection.  

Since all of the restriction maps in $M$ are surjections, the map $\mZ/I_a\ra M$ is a surjection in each spot.  Thus, it is an isomorphism.
\end{proof}

\begin{remark}
The above result depends heavily on the transfer maps in $\mZ/I_\cF$ being injections.  This does not always hold over non-cyclic groups.  For example, take $G=C_2\times C_2$ and let $\cF$ be the family $\{e\times e, e\times C_2\}$.  Then if $J=C_2\times e$, the transfer map $\mZ/\cF(G/J)\ra \mZ/\cF(\pt)$ is the zero map $\Z/2\ra \Z/2$.    
\end{remark}

We will use Proposition~\ref{pr:Z/F_a-identify} often.  Here is our first application:

\begin{prop}
\label{pr:Hom-identify}
Let $a$ and $b$ be divisors of $n$.  Then there is an isomorphism 
\[ \uHom(\mZ/I_a,\mZ/I_b)\iso \mZ/I_{(a,b)}.
\]
\end{prop}

\begin{proof}
Proposition~\ref{pr:Hom-Ann}(d) identifies $\uHom(\mZ/I_a,\mZ/I_b)$ with $\Ann_{\cF_a}(\mZ/I_b)$.  Denote this annihilator by $A$ for short.  Since $A\subseteq \mZ/I_b$ we know that $A(\Theta_1)$ is cyclic.

Fix a divisor $c$ of $n$, and consider a map $\Theta_e\ra \Theta_c$ (so $c|e$) where $e$ is a multiple of $a$.  
Then we know $\mZ/I_b(\Theta_c)=\Z/\tfrac{b}{(b,c)}$, $\mZ/I_b(\Theta_e)=\Z/\tfrac{b}{(b,e)}$, and the restriction map $\mZ/I_b(\Theta_c)\ra \mZ/I_b(\Theta_e)$ is the natural surjection.  So the kernel is generated by $\frac{b}{(b,e)}$.  Taking this together for all possible $e$, we find that   $A(\Theta_c)$ is generated by $\frac{b}{(b,[a,c])}\in \Z/\tfrac{b}{(b,c)}$.  In particular, we have that $A(\Theta_{(a,b)})=0$ and $A(\Theta_1)\iso \Z/(a,b)$. 

When $c|d$ we will show that the restriction map $A(\Theta_c)\ra A(\Theta_d)$ is surjective.  Regarding the target as contained in $\Z/\frac{b}{(b,d)}$ it is generated by $\tfrac{b}{(b,[a,d])}$, whereas the image is generated by $\frac{b}{(b,[a,c])}$.  But these generate the same subgroup; this amounts
to the numerical identity
\begin{equation}
\label{eq:exotic2}
\Bigl ( \tfrac{b}{(b,[a,c])} , \tfrac{b}{(b,d)} \Bigr ) =
\tfrac{b}{(b,[a,c,d])}=\tfrac{b}{(b,[a,d])}=\Bigl ( \tfrac{b}{(b,[a,d])} , \tfrac{b}{(b,d)} \Bigr )
\end{equation}
(see Remark~\ref{re:gcd} for how to prove such identities).  

We have now shown that $A$ satisfies all of the conditions of Proposition~\ref{pr:Z/F_a-identify}, hence $A\iso \mZ/I_{(a,b)}$.  
\end{proof}

Now suppose $a|n$ and $b|n$.  In the group $(\mZ/I_a\oplus \mZ/I_b)(\Theta_1)\iso \Z/a\oplus \Z/b$ consider the element $v=[\frac{a}{(a,b)} 1_{\Theta_1} , -\frac{b}{(a,b)} 1_{\Theta_1}]$.  The additive order of $v$ is precisely the gcd $(a,b)$.  Let $M=\langle v\rangle \subseteq \mZ/I_a\oplus \mZ/I_b$ be the submodule generated by $v$.  Then each $M(\Theta_d)$ is cyclic, all the restriction maps are surjections, and $M(\Theta_1)\iso \Z/(a,b)$.  
Since $[\mZ/I_a\oplus \mZ/I_b](\Theta_{(a,b)})\iso \Z/(\frac{a}{(a,b)})\oplus \Z/(\frac{b}{(a,b)})$, we find that the image of $v$ here is zero, hence $M(\Theta_{(a,b)})=0$.  So by Proposition~\ref{pr:Z/F_a-identify} we conclude that $M\iso \mZ/I_{(a,b)}$ and that the unique map $\mZ/I_{(a,b)}\ra M$ sending $1_{\Theta_1}$ to $v$ is an isomorphism.  

Let $Q=(\mZ/I_a\oplus \mZ/I_b)/M$.  All of the restriction  maps in $\mZ/I_a\oplus \mZ/I_b$ are surjections, so the same is true of $Q$.  Observe that $Q(\Theta_1)$ is isomorphic to $\Z/[a,b]$, by construction.  Finally, $\mZ/I_a$ and $\mZ/I_b$ both vanish at the $\Theta_{[a,b]}$ spot and so $Q$ does as well. Therefore Proposition~\ref{pr:Z/F_a-identify} shows that $Q\iso \mZ/I_{[a,b]}$.
We record the result of this analysis for later use:

\begin{prop}
\label{pr:mZ-ses}
There is a short exact sequence 
\[ 0\ra \mZ/I_{(a,b)}\ra \mZ/I_a\oplus \mZ/I_b \ra \mZ/I_{[a,b]}\ra 0
\]
that when evaluated at the $\Theta_1$ spot is the standard short exact sequence of abelian groups $0\ra \Z/(a,b)\ra \Z/a\oplus \Z/b \ra \Z/[a,b]\ra 0$.  
\end{prop}

\begin{remark}
\label{re:decompose}
If $(a,b)=1$ then the above result yields $\mZ/I_{ab}\iso \mZ/I_a\oplus \mZ/I_b$.  Consequently, every $\mZ/I_d$ breaks up as a direct sum of terms $\mZ/I_{\ell_i^{e_i}}$ where $d=\prod_i \ell_i^{e_i}$ is the prime factorization.  For another approach to this decomposition, see Proposition~\ref{pr:misc}(c) below.
\end{remark}

\begin{remark}
While we have established several formal similarities between the objects $\mZ/I_a$ and the abelian groups $\Z/a$, it should be pointed out that one should not take this too far.  For example, it is certainly not true that the kernel or cokernel of a map $\mZ/I_a\ra \mZ/I_b$ will necessarily also be a $\mZ/I_d$ for some $d$.  Taking $n=\ell^2$, the projection $\mZ/I_{\ell^2}\ra \mZ/I_\ell$ and the map $\mZ/I_{\ell}\ra \mZ/I_{\ell^2}$ that sends $1_{\Theta_1}\mapsto \ell\cdot 1_{\Theta_1}$ give simple counterexamples.  See Section~\ref{se:more_Zmod} for more on these kernels and cokernels.
\end{remark}

\subsection{Monogenal modules and localization}

The collection of modules $\mZ/I_a$, together with $\mZ$ itself, share several common features and often arise in computations.  It is useful to have a name for this class---we will call them \mdfn{monogenal} modules.  Observe that the box product of monogenal $\mZ$-modules is again monogenal.  

Note that being monogenal is not the same as being generated by a singleton in the $\Theta_1$ spot.  For example, $\mZ/2$ has the latter property but is not monogenal.

The result below contains a miscellany of results that are useful concerning monogenal modules.
First we introduce some notation. 
If $M$ is a $\mZ$-module and $a\in \Z$, define $\Ann_a(M)\subseteq M$ to be the sub-Mackey functor consisting of all elements annihilated by $a$.  
For a positive integer $a$ and a prime $\ell$, write $a(\ell)$ for the largest power of $\ell$ that divides $a$.  So $a=\prod_{\ell} a(\ell)$, where the product runs over all primes.  Finally, 
say that a $\mZ$-module is ``monogenal after $\ell$-localization'' if its $\ell$-localization is  isomorphic to the $\ell$-localization of a monogenal $\mZ$-module. 
\newpage

\begin{prop}
\label{pr:misc} Let $M$ be a $\mZ$-module.
\begin{enumerate}[(a)]
\item Suppose that $(a,b)=1$ and $ab$ annihilates $M$.  Then $M=\Ann_a(M)\oplus \Ann_b(M)$.
\item Suppose $a|n$, $b|n$, and $(a,b)=1$.  Then $\Ann_a(\mZ/I_{ab})\iso \mZ/I_a$.  
\item If $a|n$ and $b|n$ and $(a,b)=1$ then $\mZ/I_a\oplus \mZ/I_b\iso \mZ/I_{ab}$.
\item If $\ell$ is a prime and $a|n$ then $(\mZ/I_{a})_{(\ell)}\iso \mZ/I_{a(\ell)}$. 
\item If $M$ and $N$ are monogenal $\mZ$-modules that are isomorphic, then any surjection $M\ra N$ is an isomorphism.
\item Suppose that $M$ is finitely-generated.
If $M$ is monogenal after $\ell$-localization, for each prime $\ell$, then $M$ is monogenal.  If a subset $S\subseteq M(\pt)$ generates $M_{(\ell)}$ for each prime $\ell$, then $S$ generates $M$.
\end{enumerate}
\end{prop}

\begin{proof}
These are mostly elementary.  Part (a) follows from the fact that $1$ is a linear combination of $a$ and $b$.  For (b) we use the recognition conditions from Proposition~\ref{pr:Z/F_a-identify}.  Set $W=\mZ/I_{ab}$ for convenience.  Each abelian group $W(\Theta_r)$ is cyclic, hence its $a$-annihilator is also cyclic.  We have $W(\Theta_r)\iso \Z/(\frac{ab}{(r,ab)})$ and so $W(\Theta_1)\iso \Z/(ab)$ and $W(\Theta_a)\iso\Z/b$.  Hence $(\Ann_a W)(\Theta_1)\iso \Z/a$ and $(\Ann_a W)(\Theta_a)=0$.  Finally, if $X\ra Y$ is a surjection of abelian groups annihilated by $ab$ it is elementary to check that $\Ann_a(X)\ra \Ann_a(Y)$ is still surjective (use the decomposition as in (a)).  So all of the restriction maps in $\Ann_a(W)$ are surjective. Now apply Proposition~\ref{pr:Z/F_a-identify}.  

Part (c) follows directly from (a) and (b).

For (d), observe that (c) implies $\mZ/I_a\iso \bigoplus_{q} \mZ/I_{a(q)}$ where $q$ ranges over all primes (and note that this is a finite direct sum).  Since $\mZ/I_{a(q)}$ is annihilated by a power of $q$, its $\ell$-localization is zero for $q\neq \ell$ and is equal to itself for $q=\ell$.  

For (e) just look at the map $M(\Theta)\ra N(\Theta)$ for every $C_n$-orbit $\Theta$ and use the related property of cyclic abelian groups.  

For (f), first note that a finitely-generated abelian group that is cyclic after $\ell$-localization for all primes $\ell$ is necessarily cyclic.  Likewise, a map of abelian groups that is surjective after localization at each prime is necessarily surjective.  So each $M(\Theta_d)$ is cyclic, and every restriction map in $M$ is surjective.  

Suppose $M(\Theta_1)\iso \Z$, and pick a generator $g$.  It follows that $M(\Theta_d)$ must also be isomorphic to $\Z$ for every $d|n$, generated by $p^*(g)$.  So the map $\mZ\ra M$ that sends $1$ to $g$ is surjective, and hence an isomorphism.  

If $M(\Theta_1)\not\iso \Z$ then 
$M(\Theta_1)$ is torsion, and by surjectivity of the restriction maps each $M(\Theta_d)$ is torsion as well.
Pick an integer $s$ that annihilates $M$ and write $s=\ell_1^{e_1}\cdots \ell_k^{e_k}$ where the $\ell_i$ are distinct primes.  Then by (a) we have $M=\oplus_i \Ann_{\ell_i^{e_i}}(M)$, and therefore $M_{(\ell_i)}=\Ann_{\ell_i^{e_i}}(M)$.  By assumption this is monogenal, so it is isomorphic to $\mZ/I_{\ell_i^{f_i}}$ for some $f_i$.    Now use (c) to conclude that $M$ is monogenal.  

The final claim is just standard commutative algebra.  Let $N\subseteq M$ be the subMackey functor generated by $S$.  Then for each $a|n$ we have $N(\Theta_a)\subseteq M(\Theta_a)$, and the subset is an equality after localization at every prime---hence, it is an equality.  Since this holds for all $a$, $N=M$.  That is, $S$ generates $M$.
\end{proof}

\subsection{The resolution for \mdfn{$\mZ/I_a$}}

Consider the sequence of maps 
\begin{equation} 
\label{eq:complex}
\xymatrix{
F_{\pt} \ar[r]^{Ip} & F_{\Theta_a} \ar[r]^{\id-Rt} & F_{\Theta_a} \ar[r]^{Rp} & F_{\pt}. }
\end{equation}
One checks that this is a chain complex exactly as in Example~\ref{ex:cyclic-res}.
In terms of generators, the map $Rp$ sends $g_{\Theta_a}$ to $p^*(g_{\pt})=1_a\in \mZ(\Theta_a)$.  So the image of $Rp$ is precisely $I_a$, and the cokernel is $\mZ/I_a$.  

\begin{remark}
\label{re:complex-iso}
There is an evident lack of symmetry in the way the complex is written in (\ref{eq:complex}), but one should observe the isomorphism of complexes
\[ \xymatrix{
0 \ar[r] & F_{\pt} \ar[d]_{\id} \ar[r]^-{Ip} & F_{\Theta_a} \ar[d]_{\id} \ar[r]^-{\id-Rt} & F_{\Theta_a} \ar[d]^{-Rt^{-1}}\ar[r]^-{Rp} & F_{\pt}\ar[d]^{-\id}\ar[r] & 0 \\
0 \ar[r] & F_{\pt} \ar[r]^{Ip} & F_{\Theta_a} \ar[r]^{\id-It} & F_{\Theta_a} \ar[r]^-{Rp} & F_{\pt}\ar[r] & 0 \\
}
\]
which uses that $It=Rt^{-1}$ (Proposition~\ref{pr:I=R}) and $p\circ t^{-1}=p$.  Isomorphisms of complexes such as this one---translating between an $Rt$ and an $It$---show up frequently in what follows and will often be made implicitly.
\end{remark}

For the case where $n$ is a prime power, incarnations of the following result can be found in \cite[Section 3.3]{Z} and \cite[Proposition 3.3]{HHR2}.

\begin{prop}
\label{pr:Z/Fa-res}
Fix $n\geq 1$.  For any $a|n$, the sequence
\[ \xymatrix{
0 \ar[r] & F_{\pt} \ar[r]^{Ip} & F_{\Theta_a} \ar[r]^{\id-Rt} & F_{\Theta_a} \ar[r]^{Rp} & F_{\pt} \ar[r] & \mZ/I_a \ar[r] & 0}
\]
is exact.  Consequently, this is a free resolution of $\mZ/I_a$.  From now on we will refer to it as the \dfn{standard free resolution} for $\mZ/I_a$.  
\end{prop}

\begin{proof}
For any $c|n$ we need to evaluate the above complex at $\Theta_c$ and verify that we get a sequence of abelian groups that is exact in degrees $1$, $2$, and $3$ (where the rightmost $F_{\pt}$ is given degree $0$, as usual).  As in Remark~\ref{re:translate}
this complex is isomorphic to
\[ \xymatrixcolsep{2.7pc}\xymatrix{
0 \ar[r] & \mZ(\pt\times \Theta_c) \ar[r]^{(\pi \times \id)^*} & \mZ(\Theta_a\times \Theta_c) \ar[r]^{\id-(t\times \id)_*} & \mZ(\Theta_a\times \Theta_c) \ar[r]^{(\pi \times \id)_*} & \mZ(\pt\times \Theta_c)
}
\]
and we need to verify exactness at the three spots on the left.
To understand this sequence we will make use of an isomorphism
\[ \Theta_a\times \Theta_c \iso \coprod_{i=0}^{(a,c)-1}\Theta_{[a,c]}.
\]
To describe this isomorphism let us introduce a placeholder $\sigma^i$ to denote the $i$th copy of $\Theta_{[a,c]}$ in the coproduct.  We have our canonical projections $\Theta_{[a,c]}\ra \Theta_a$ and $\Theta_{[a,c]}\ra \Theta_c$, and we will write $x$ both for an element of the domain and its image in the codomain. Under these conventions, define
\[ F\colon \coprod_{i=0}^{(a,c)-1} \sigma^i \Theta_{[a,c]}\lra \Theta_a\times \Theta_c
\]
by $\sigma^ix\mapsto (x,t^ix)$.  This induces an isomorphism of $C_n$-sets, and hence an isomorphism
\[ \Z^{(a,c)}= \bigoplus_i \mZ(\Theta_{[a,c]}) \llra{F_*} \mZ(\Theta_a\times \Theta_c)
\]
where in the first part we have used that $\mZ(\Theta_{[a,c]})$ has the canonical generator $1$.
Since $F$ is an isomorphism, note that $F_*=(F^*)^{-1}$ (Proposition~\ref{pr:I=R}).

Observe the existence of commutative diagrams
\[ \xymatrix{
\coprod_i \sigma^i\Theta_{[a,c]} \ar[r]^-F \ar[d]_\alpha  & \Theta_a\times \Theta_c \ar[d]^{t\times \id} && 
\coprod_i \sigma^i\Theta_{[a,c]} \ar[r]^-F \ar[dr]_\beta  & \Theta_a\times \Theta_c \ar[d]^{\pi_2=\pi\times \id}
\\
\coprod_i \sigma^i \Theta_{[a,c]} \ar[r]^-F & \Theta_a\times \Theta_c
&&
 & \Theta_c
}
\]
where 
$\beta(\sigma^ix)=p(t^ix)$ and
\[ \alpha(\sigma^ix)=\begin{cases}
\sigma^{i-1}(tx) & \text{when $i>0$},\\
\sigma^{(a,c)-1}(t^\epsilon x) & \text{when $i=0$}
\end{cases}
\]
and here $\epsilon$ is the unique integer modulo $[a,c]$ having the property that $\epsilon \equiv 1$ (mod $a$) and $\epsilon \equiv 1-(a,c)$ (mod $c$). Note we are using the pullback diagram
\[ \xymatrix{\Z/{[a,c]}\ar@{->>}[r]\ar@{->>}[d] & \Z/c \ar@{->>}[d] \\
\Z/a \ar@{->>}[r] & \Z/{(a,c)}
}
\]
to define $\epsilon$.  

The maps $F$, $\alpha$, and $\beta$ yield an isomorphism of chain complexes as follows:
\[ \xymatrixcolsep{2.7pc}\xymatrix{
0 \ar[r] & \mZ(\pt\times \Theta_c)  \ar[r]^{(\pi\times \id)^*} & \mZ(\Theta_a\times \Theta_c) \ar[r]^{\id-(t\times \id)_*}  & \mZ(\Theta_a\times \Theta_c) \ar[r]^{(\pi\times \id)_*}  & \mZ(\pt\times \Theta_c) \\
0 \ar[r] & \Z \ar[r]^C\ar[u]_\iso & 
\bigoplus_i \sigma^i\Z \ar[r]^B\ar[u]_\iso^{(F^*)^{-1}=F_*} & \bigoplus_i \sigma^i \Z\ar[r]^A\ar[u]_\iso^{F_*} & \Z\ar[u]_\iso
}
\]
where $A(\und{x})=\tfrac{[a,c]}{c}\sum x_i$, $B(\und{x})=(x_0-x_1,x_1-x_2,\ldots,x_{(a,c)-1}-x_0)$, and $C(1)=(1,1,\ldots,1)$.  Note that $A=\beta_*$, $B=\id-\alpha_*$, and $C=\beta^*$.  The lower sequence is readily checked to be exact. 
\end{proof}

The following result is almost immediate:

\begin{prop}
\label{pr:Ext-Mackey0}
Fix $n\geq 1$ and let $a|n$.  Let $M$ be a $\mZ$-module for $C_n$.  Then for any $c|n$ the value $\uExt^i(\mZ/I_a,M)(\Theta_c)$ is isomorphic to the $i$th cohomology of the complex
\[
\xymatrixcolsep{2.2pc}\xymatrix{
0 & M(\Theta_c)\ar[l] & M(\Theta_{[a,c]}) \ar[l]_-{p^*}
&&
M(\Theta_{[a,c]}) \ar[ll]_{\id-(t^*)^{\epsilon+(a,c)-1}} & M(\Theta_c) 
\ar[l]_-{p_*} & 0 \ar[l]
}
\]
where the rightmost $M(\Theta_c)$ is in cohomological degree $0$ and where $\epsilon$ is the unique integer modulo $[a,c]$ such that $\epsilon \equiv 1$ (mod $a$) and $\epsilon\equiv 1-(a,c)$ (mod $c$).
\end{prop}

\begin{proof}
One applies $\uHom(\blank,M)$ to the resolution from Proposition~\ref{pr:Z/Fa-res} and evaluates at $\Theta_c$.  Using Remark~\ref{re:translate} this yields the sequence
\[ \xymatrixcolsep{2.25pc}\xymatrix{
0 & M(\pt \times \Theta_c) \ar[l] & M(\Theta_a\times \Theta_c) \ar[l]_{(\pi\times \id)_*} & M(\Theta_a \times \Theta_c) \ar[l]_{\id-(t\times \id)^*} & M(\pt \times \Theta_c). \ar[l]_{(\pi\times \id)^*}
}
\]
Next one uses the maps $F$, $\alpha$, and $\beta$ from the above proof in exactly the same way as there, which leads to the top complex in the following diagram:

\[
\xymatrixcolsep{1.6pc}\xymatrix{
0 & M(\Theta_c)\ar[l] & \bigoplus_{i=0}^r \sigma^i M(\Theta_{[a,c]}) \ar[l]_-{S}
&
\bigoplus_{i=0}^{r} \sigma^i M(\Theta_{[a,c]}) \ar[l]_-T & M(\Theta_c)\ar[dl]^{p^*} 
\ar[l]_-U & 0. \ar[l]\\
 && M(\Theta_{[a,c]}) \ar[u]^Q \ar[ul]^{p_*}
& M(\Theta_{[a,c]}) \ar[u]^R \ar[l]_{\id-(t^*)^{\epsilon+(a,c)-1}} 
}
\]
Here the rightmost $M(\Theta_c)$ is in cohomological degree $0$, $r=(a,c)-1$, and the maps are given by
\begin{align*}
S(y_0,\ldots,y_r)&=p_*y_0 + p_*t_*y_1+p_*(t_*)^2y_2+\cdots+p_*(t_*)^ry_r , \\
T(y_0,\ldots,y_r)&=(y_0-(t^*)^\epsilon y_r,\, y_1-t^*y_0, \, y_2-t^*y_1,\, \ldots,\, y_r-t^*y_{r-1}), \\
U(z)&= (p^*z,p^*(t^*z),p^*((t^*)^2z),\ldots,p^*((t^*)^rz)).
\end{align*}

Finally, defining the vertical maps by
\[ Q(x)=(x,0,\ldots,0), \quad R(x)=(x,t^*x,(t^*)^2x,\ldots,(t^*)^rx)
\]
makes the diagram commute and gives a quasi-isomorphism of chain complexes, since it is readily checked that the kernel and cokernel of $T$ are the same as those of $\id-(t^*)^{\epsilon+(a,c)-1}$.  
\end{proof}

\begin{prop}
\label{pr:Ext-Mackey}
Fix $n\geq 1$, and let $a$ and $b$ be divisors of $n$.  Then $\mExt^i(\mZ/I_a,\mZ/I_b)\iso\begin{cases} \mZ/I_{(a,b)} & \text{if $i=0,3$},\\ 0 & \text{otherwise.}\end{cases}$ 
\end{prop}

\begin{remark}
Note from (\ref{eq:uHom=Hom})  that evaluating the $\uExt$ Mackey functors at $\pt$ yields that 
\[ \Ext^i(\mZ/I_a,\mZ/I_b)\iso\begin{cases} \Z/(a,b) & \text{if $i=0,3$},\\ 0 & \text{otherwise.}\end{cases}
\]
\end{remark}

\begin{proof}[Proof of Proposition~\ref{pr:Ext-Mackey}]
Let $M$ be a $\mZ$-module where all of the $t^*$ maps are the identity.  By Proposition~\ref{pr:Ext-Mackey0} 
for any $c|n$ the group
$\uExt^i(\mZ/I_a,M)(\Theta_c)$ is computed by the complex  
\[ \xymatrix{
0 & M(\Theta_c) \ar[l] & M(\Theta_{[a,c]})  \ar[l]_-{p_*} &  M(\Theta_{[a,c]}) \ar[l]_-{0} & M(\Theta_c) \ar[l]_-{p^*} & 0. \ar[l]
}
\]
Now take $M=\mZ/I_b$. Then the restriction maps are all surjective and the transfer maps are all injective, so this immediately yields 
$\mExt^i(\mZ/I_a,\mZ/I_b)=0$ for $i=1,2$. We next use the recognition theorem from Proposition~\ref{pr:Z/F_a-identify} to compute $\mExt^3$.

For $M=\mZ/I_b$ the above complex takes the form
\[ \xymatrix{
0 & \Z/r \ar[l] &  \Z/s \ar[l]_-S &  \Z/s \ar[l]_-0 & \Z/r \ar[l]_-U & 0 \ar[l]
}
\]
where $r=\tfrac{b}{(b,c)}$, $s=\tfrac{b}{(b,[a,c])}$, and the maps are $U(x)=\bar{x}$, $S(x)=\tfrac{r}{s}\cdot x$.  Here $x\mapsto \bar{x}$ is the canonical projection $\Z/r\ra \Z/s$ (note that $s|r$). In particular, we find that
$\mExt^3(\mZ/I_a,\mZ/I_b)(\Theta_{(a,b)})=0$ and
$\mExt^3(\mZ/I_a,\mZ/I_b)(\Theta_1)\iso \Z/(a,b)$. 
Finally, observe that the proof of Proposition~\ref{pr:Ext-Mackey0} exhibits $\uExt^3(\mZ/I_a,M)$ as a quotient of $M$, for any $M$.  Since the restriction maps in $\mZ/I_b$ are all surjections, the same is true for $\uExt^3(\mZ/I_a,\mZ/I_b)$.  It therefore follows from Proposition~\ref{pr:Z/F_a-identify} that $\uExt^3(\mZ/I_a,\mZ/I_b)\iso \mZ/I_{(a,b)}$, as desired. 

The $\mExt^0$ computation can be done in this way too, but in fact we already did it from a different perspective back in Proposition~\ref{pr:Hom-identify}.
\end{proof}

\begin{remark}
More generally, it follows from Proposition~\ref{pr:Ext-Mackey0} that for any $\mZ$-module $M$ one has $\mExt^3(\mZ/I_a,M)\iso M/M\{\cF_a\}=M\bbox \mZ/I_a$, and of course $\mExt^0(\mZ/I_a,M)\iso \Ann_{\cF_a}(M)$.  We leave the proofs to the reader.  However, we do not know  nice descriptions for the $\mExt^1$ and $\mExt^2$ computations for general $M$. 
\end{remark}

\begin{prop}
\label{pr:Ext-Z}
We have $\uExt^i(\mZ/I_a,\mZ)\iso \begin{cases}
\mZ/I_a & \text{if $i=3$},\\
0 & \text{otherwise}.  
\end{cases}$
\end{prop}

\begin{proof}
This is an easy computation using the same analysis as in the proof of Proposition~\ref{pr:Ext-Mackey}.  Or one can observe that the resolution of $\mZ/I_a$ given in Proposition~\ref{pr:Z/Fa-res} is self-dual, in the sense that applying $\uHom(\blank,\mZ)$ gives a complex isomorphic to the original resolution (see Remark~\ref{re:complex-iso}).  
\end{proof}

We next turn to the analogous $\uTor$ computations.  These can be done in exactly the same way we did the $\uExt$ computations, but in fact we get them for free using duality of the standard resolution:

\begin{prop}
\label{pr:Ext-Tor}
There are isomorphisms, functorial in $M$, of the form
\[ \uTor_i(\mZ/I_a,M)\iso \uExt^{3-i}(\mZ/I_a,M)
\]
for all $i$.
\end{prop}

\begin{proof}
Let $J_\bullet \ra \mZ/I_a$ be the standard resolution.  We have an isomorphism of complexes $\uHom(J_\bullet,M)\iso \uHom(J_\bullet,\mZ)\bbox M$
by Corollary~\ref{co:box-hom}. Using Proposition~\ref{pr:box-slice-isos} we have that $\uHom(J_\bullet,\mZ)$ is isomorphic to the complex
\[
\xymatrix{
0 & F_{\pt} \ar[l] & F_{\Theta_a} \ar[l]_{Rp} & F_{\Theta_a} \ar[l]_{\id-It} & F_{\pt} \ar[l]_{Ip} & 0 \ar[l]
}
\]
where the right-most $F_{\pt}$ is in degree $0$. Thus $\uHom(J_\bullet,\mZ)\iso \Sigma^{-3}J_\bullet$, as shown in Remark~\ref{re:complex-iso}.
This yields the desired isomorphisms of homology groups.
\end{proof}

\begin{prop}
\label{pr:Tor-Mackey1}
Let $a$ and $b$ be divisors of $n$. Then 
\[
\uTor_i(\mZ/I_a, \mZ/I_b)\cong \begin{cases}\mZ/I_{(a,b)} & \text{ if } i=0,3, \\ 0 & \text{otherwise}.\end{cases}
\]
\end{prop}

\begin{proof}
Immediate.
\end{proof}

\subsection{A few more computations}
\label{se:more-computations}

One sometimes needs to know $\uTor_*(I_a,I_b)$, $\uTor_*(I_a,\mZ/I_b)$, or perhaps the analogs with $\uExt$.  These can be computed using the short exact sequences $0\ra I_b\ra \mZ \ra \mZ/I_b\ra 0$ together with the calculations we have already done.  Here are the answers:

\begin{prop}
\label{pr:Tor-I}\mbox{}\par
\begin{itemize}
\item 
$ \uExt^{3-i}(\mZ/I_a,I_b) \iso \uTor_i(\mZ/I_a,I_b)\iso \begin{cases}
\mZ/I_{(a,b)} & \text{if $i=2$},\\
I_{(a,b)}/I_a & \text{if $i=0$}, \\
0 & \text{otherwise,}
\end{cases}$
\item $\uTor_i(I_a,I_b)\iso \begin{cases}
\mZ/I_{(a,b)} & \text{if $i=1$},\\
I_{[a,b]} & \text{if $i=0$},\\
0 & \text{otherwise.}
\end{cases}$
\end{itemize}
In particular, note the isomorphisms
\[ \mZ/I_a\bbox I_b\iso I_{(a,b)}/I_a \qquad\text{and}\qquad
I_a\bbox I_b\iso I_{[a,b]}.
\]
\end{prop}

\begin{proof}
The duality between $\uExt$ and $\uTor$ is from Proposition~\ref{pr:Ext-Tor}.  The computations of the $\uTor_i$ for $i>0$ is immediate from the long exact $\Tor$-sequence.  It  only remains to do the $\uTor_0$ computations.

For any $\mZ$-module $M$ one formally has $M\bbox \mZ/I_a\iso M/(I_a\cdot M)$, and so $I_b\bbox \mZ/I_a\iso I_b/(I_a\cdot I_b)$.  But Corollary~\ref{co:I-product} identifies $I_a\cdot I_b$ with $I_{[a,b]}$, and also gives $I_b/(I_a\cdot I_b)=I_b/I_{[a,b]}\iso I_{(a,b)}/I_a$.  

Finally, take the exact sequence
 $0\ra I_b\ra \mZ \ra \mZ/I_b\ra 0$ and box with $I_a$ to get $0\ra I_a\bbox I_b \ra I_a \ra I_a\bbox \mZ/I_b\ra 0$.  This sequence is still exact because $\uTor_1(I_a,\mZ/I_b)=0$.  It follows that $I_a\bbox I_b$ is isomorphic to its image in $I_a$, which is precisely $I_a\cdot I_b=I_{[a,b]}$.
\end{proof}

\subsection{Homological dimension}

All of the homological computations we have done so far have vanished in degrees larger than three, and this is not a coincidence.  As explained in \cite{Z}, the following result is proven in \cite{BSW}.  As the statement is not totally transparent in \cite{BSW}, we include a sketch of the proof.

\begin{prop}
When $G=C_n$, every finitely-generated $\mZ$-module has a projective resolution of length three or less.
\end{prop}

\begin{proof}
The argument is taken from \cite[Proof of Proposition 7.1]{BSW}.  Let $M$ be a finitely-generated $\mZ$-module.  Then there is a surjection $f_0\colon F_0\ra M$ where $F_0$ is a finite direct sum of free modules, and likewise there is a surjection $F_1\ra \ker(f_0)$ where $F_1$ is a finite direct sum of free modules.  By Proposition~\ref{pr:fixedpointfree} we can write $F_0=\FP(W_0)$ and $F_1=\FP(W_1)$ where $W_0$ and $W_1$ are finitely-generated permutation modules.  The composite $F_1\ra \ker(f_0)\inc F_0$ is, by the $(\ev_G,\FP)$-adjunction, adjoint to a map $W_1=\FP(W_1)(G)\ra W_0$.  This implies that $F_1\ra F_0$ equals the map $\FP(W_1\ra W_0)$.

Let $K=\ker(W_1\ra W_0)$. 
Then $0\ra \FP(K)\ra \FP(W_1)\ra \FP(W_0)$ is exact, by inspection. Now we come to the key point, which is that \cite[Corollary 4.8]{A} shows that for any finitely-generated $\Z[C_n]$-module $K$ the $\mZ$-module $\FP(K)$ has a projective resolution of length at most one.  Splicing this into the $F_1\ra F_0\ra M\ra 0$ sequence shows that $M$ has a projective resolution of length at most three.
\end{proof}


\section{More \mdfn{$\mZ$}-modules}
\label{se:more_Zmod}

So far we have introduced the $\mZ_{C_n}$-modules $I_d$ and $\mZ/I_d$, for $d|n$.  In this section we introduce a ``next round'' of useful actors that includes the quotients $I_a/I_b$, which we already saw arise in Proposition~\ref{pr:Tor-I}.  Unlike the category of abelian groups, there is (in general) no reasonable description of the finitely-generated indecomposable $\mZ$-modules---in a certain technical sense the module theory is ``wild''.  So we should not expect to ever be done with this process. The more computations one does, the more new modules one is likely to encounter.

\subsection{Torsion modules}
Fix $b$ and $c$ that divide $n$.  Maps $\mZ/I_c\ra \mZ/I_b$ are uniquely determined by the image of $1_{\pt}$ in $\mZ/I_b(\pt)=\Z/b$, which can be any element $x$ annihilated by $c$.   This is the condition that $b|cx$, and the map $\mZ/I_c\ra \mZ/I_b$ can be considered to be multiplication-by-$x$.  The cokernel and kernel of this map will be denoted as follows:
\[ \mZ/(I_b,x)=\coker\bigl ( \mZ/I_c\llra{x} \mZ/I_b\bigr ) \qquad\text{and}\qquad
\tfrac{I_b\colon x}{I_c}=\ker \bigl(\mZ/I_c \llra{x} \mZ/I_b\bigr ).
\]
To explain the latter notation, recall the colon ideal $I_b\colon x=\{y\in \mZ\,|\, xy\in I_b\}$ (here we are regarding $x$ as an integer).  One can see that the assumption $b|cx$ implies that $I_c\subseteq I_b\colon x$, though this is not exactly trivial.  Recalling that in spot $\Theta_d$ one has $I_b(\Theta_d)=(\tfrac{b}{(b,d)})$ and $I_c(\Theta_d)=(\tfrac{c}{(c,d)})$, the statement boils down to $b|cx \Rightarrow \tfrac{b}{(b,d)}| \tfrac{c}{(c,d)}x$.  This final statement can be verified by examining the $\ell$-adic valuation for each prime $\ell$.

Note that when $x=1$ (so $b|c$) one has $I_b\colon x=I_b$ and the short exact sequence
\[ 0 \ra I_b/I_c \ra \mZ/I_c \ra \mZ/I_b \ra 0.
\]
If $b|c$ then we also have the canonical inclusion $\mZ/I_b \subseteq \mZ/I_c$ (corresponding to multiplication by $\frac{c}{b}$), yielding the short exact sequence
\[ 0 \ra \mZ/I_b \inc \mZ/I_c \ra \mZ/(I_c,\tfrac{c}{b})\ra 0.
\]

\begin{prop}
Let $b,c,d$ be divisors of $n$ and suppose $b|cx$.  
\begin{enumerate}[(a)]
\item $[\mZ/(I_b,x)](\Theta_d)$ is cyclic of order $(\tfrac{b}{(b,d)},x)$.
\item $[\tfrac{I_b\cln x}{I_c}](\Theta_d)$ is cyclic of order $\tfrac{(bc,cxd)}{(bc,bd)}$.
\end{enumerate}
In particular, $[\mZ/(I_b,x)](\Theta_1)\iso \Z/(b,x)$ and 
$[\tfrac{I_b\cln x}{I_c}](\Theta_1)\iso \Z/\tfrac{(bc,cx)}{b}$.  Also, if $b|c$ then $I_b/I_c(\Theta_1)\iso \Z/(\frac{c}{b})$.  
\end{prop}

\begin{proof}
The statements about $\mZ/(I_b,x)$ are immediate from Proposition~\ref{pr:Z/F_a-props}.  For (b) use that the kernel of $\Z/P\llra{x}\Z/Q$ for any integers $P,Q$ is cyclic of order $\frac{P(Q,x)}{Q}$. In our case $P=\frac{c}{(c,d)}$ and $Q=\frac{b}{(b,d)}$ and observe that
\[ \frac{Q}{(Q,x)}= \frac{\tfrac{b}{(b,d)}}{(\tfrac{b}{(b,d)},x)}=\frac{b}{(b,x(b,d))}=\frac{b}{(b,xd)}.\]
Finally,
\[
P \cdot \frac{(Q,x)}{Q}=\frac{c}{(c,d)}\cdot \frac{(b,xd)}{b}=\frac{(bc,cxd)}{(bc,bd)}.
\]
\end{proof}

Describing the restriction and transfer maps in $\mZ/(I_b,x)$ and $\tfrac{I_b\cln x}{I_c}$ is complicated and we will not pursue it.  

The $\mZ$-modules we have just introduced arise in $\uExt$ and $\uTor$ computations involving the $I_a/I_b$ modules:

\begin{prop}
\label{pr:Ia/Ib}
Let $b,c,d$ be divisors of $n$ and assume $b|c$.
\begin{enumerate}[(a)]
\item $\uExt^{3-i}(\mZ/I_d,I_b/I_c)\iso \uTor_i(I_b/I_c,\mZ/I_d)\iso \begin{cases}
\frac{I_{(b,d)}\colon \frac{c(b,d)}{b(c,d)}}{I_{(c,d)}} & \text{if $i=3$}, \\
\mZ/(I_{(b,d)},\frac{c(b,d)}{b(c,d)}) & \text{if $i=2$},\\
0 & \text{if $i=1$, $i\geq 4$}, \\
I_{(b,d)}/I_{(c,d)} & \text{if $i=0$}.
\end{cases}$
\item $\uExt^i(I_b/I_c,\mZ/I_d)\iso \begin{cases}
\mZ/(I_{(c,d)},\tfrac{c}{b}) & \text{if $i=3$},\\
\frac{I_{(c,d)}\cln \tfrac{c}{b}}{I_{(b,d)}} & \text{if $i=2$},\\
0 & \text{if $i=1$ or $i\geq 4$}, \\
\mZ/(I_{(c,d)},\tfrac{(c,d)}{(b,d)}) & \text{if $i=0$}.
\end{cases}$
\end{enumerate}
\end{prop}

\begin{proof}
The proof is just a calculation with the long exact sequences induced by $0\ra I_b/I_c \ra \mZ/I_c \ra \mZ/I_b\ra 0$.  The key step is lifting the projection $\mZ/I_c\ra \mZ/I_b$ to a map of resolutions and observing that this lift is multiplication by $\frac{c}{b}$ in degree $3$. Part (b) follows readily from this observation.

For the $\uTor$ part of (a) one needs to calculate the map on $\uTor_3(\blank,\mZ/I_d)$ induced by $\mZ/I_c\ra \mZ/I_b$. This comes down to analyzing the $\pt$-spot after boxing with $\mZ/I_d$, which gives the diagram
\[\xymatrix{
\Z/d \ar[r]^{\cdot \frac{c}{b}}\ar@{->>}[d] & \Z/d \ar@{->>}[d] \\
\Z/\frac{d}{(d,c)} \ar@{->>}[r] & \Z/\frac{d}{(d,b)}.
}
\]
The kernels of the vertical maps give $\uTor_3(\mZ/I_c, \mZ/I_d)(\pt)$ and $\uTor_3(\mZ/I_c, \mZ/I_d)(\pt)$, and these are generated by $\frac{d}{(d,c)}$ and $\frac{d}{(d,b)}$, respectively. The induced map on the kernels sends 
\[ \tfrac{d}{(d,c)}\mapsto \tfrac{c}{b}\cdot \tfrac{d}{(d,c)}=\tfrac{c(b,d)}{b(c,d)}\cdot \tfrac{d}{(d,b)},
\]
so in $H_3$ the generator is sent to $\tfrac{c(b,d)}{b(c,d)}$ times the generator of the target.  

Finally in (a), the $\Ext$-$\Tor$ duality is from Proposition~\ref{pr:Ext-Tor}.
\end{proof}

\subsection{Forms of \mdfn{$\Z$}}
\label{se:forms-of-Z}
Following \cite{HHR2} and \cite{Z} it is useful
to define a ``form of $\Z$'' to be any $\mZ$-module whose value at every orbit is isomorphic to $\Z$.  As first examples, note that each $I_d$ is a form of $\Z$.  The Bredon cohomology groups $H_{C_n}^{\lambda_b-\lambda_c}(\pt)$ also turn out to be forms of $\Z$ (see Section~\ref{se:appl}), which is the main reason for exploring them.

The following two lemmas will be useful:

\begin{lemma}
Suppose $X$ is a form of $\Z$.  Then every restriction and transfer map in $X$ is injective, and every $t$-map is the identity.  If $e=d\cdot \ell$ where $\ell$ is prime and $p\colon \Theta_e\ra \Theta_d$ is the projection, then as maps $\Z\ra \Z$ both $p_*$ and $p^*$ are multiplication by elements from the set $\{\pm 1, \pm \ell\}$.  
\end{lemma}

\begin{proof}
Pick an orbit $\Theta_d$ and consider $\pi\colon \Theta_d\ra \Theta_1$.  Then $\pi_*\pi^*$ is multiplication by $d$ as a map $\Z\ra \Z$, which is 
injective.  So both $\pi^*$ and $\pi_*$ are nonzero maps between $\Z$ and thus injective.  The same idea works for every restriction and transfer pair.  

The $t^*$ map on $X(\Theta_d)$ is an automorphism from $\Z$ to itself, so it is multiplication by $1$ or $-1$.  It also must satisfy $t^*\pi^*=\pi^*$.  Since $\pi^*$ is nonzero,  $t^*$ must be multiplication by $1$.  

The last statement of the proposition follows because $p_*p^*$ is multiplication by $\ell$ and $\ell$ is prime.
\end{proof}

\begin{lemma}
\label{le:twistedZ}
Suppose $X$ is a form of $\Z$ and $0\ra X\ra Y \ra K\ra 0$ is an exact sequence where 
\begin{enumerate}[(1)]
\item the values of $K$ are all cyclic torsion groups, and
\item the transfer maps in $K$ are all injective.
\end{enumerate}
If $Y(\pt)$ is torsion-free then $Y$ is also a form of $\Z$.  (Examples of $K$ are the modules $\mZ/I_d$ or any submodule of such.)
\end{lemma}

\begin{proof}
 We evaluate the short exact sequence at the orbits $\Theta_d$ and $\Theta_1$ to get
\[ \xymatrix{
0 \ar[r] & X(\Theta_d) \ar[r]\ar[d]_{\pi_*} & Y(\Theta_d) \ar[r]\ar[d]_{\pi_*} & K(\Theta_d) \ar[d]^{\pi_*}\ar[r] & 0 \\
0 \ar[r] & X(\Theta_1) \ar[r] & Y(\Theta_1)\ar[r] & K(\Theta_1) \ar[r] & 0.
}
\]
In the top row, the left term is isomorphic to $\Z$ and the right term is $\Z/e$ for some $e$.  So if we can show that the middle group is nontorsion, it will be isomorphic to $\Z$ and we will be done.

By assumption the vertical transfer map on the right is an injection and $Y(\Theta_1)$ is non-torsion. It now follows by an easy diagram chase that $Y(\Theta_d)$ cannot have torsion, and hence it must be that $Y(\Theta_d)\iso \Z$.
\end{proof}

Let us call the projections $\Theta_{d}\ra \Theta_{c}$ where $\frac{d}{c}$ is prime the \dfn{elementary projections}, and let us call the pushforwards and pullbacks corresponding to these the \dfn{elementary maps} in a $C_n$-Mackey functor.  All of the standard projections are composites of elementary projections, therefore all of the pushforwards and pullbacks for standards projections are determined by the elementary ones.

For general $n$, suppose $e|n$ and $d|e$.  Write $e=\prod_i \ell_i^{e_i}$ and $d=\prod_i \ell_i^{d_i}$, where the $\ell_i$ are the distinct prime factors of $n$.  Define a $\mZ$-module $\mZ(e;d)$ that equals $\Z$ at every orbit and has elementary
maps defined as follows.  At an orbit $\Theta_k$ where $k=\prod_i \ell_i^{k_i}$, we have elementary restriction and transfer maps emanating from $\Theta_k$ in each $\ell_i$-direction:
\begin{itemize}
\item If $k_i\geq e_i$ or $k_i<d_i$ then the restriction map is the identity.  If $k_i \in [d_i,e_i)$ then the restriction map is multiplication by $\ell_i$.  
\item If $k_i\geq e_i+1$ or $k_i\leq d_i$ then the transfer map is multiplication by $\ell_i$.  If $k_i\in (d_i,e_i]$ then the transfer map is the identity.
\end{itemize}
This defines a $\mZ$-module that we will denote as $\mZ(e;d)$.  A helpful picture is in Figure~\ref{fig:Z(d,e)} below: depicting a $C_n$-Mackey functor as the usual rectangular grid, we have put dots at the $\Theta_e$ and $\Theta_d$ spots and then drawn the restriction maps in the different regions (in the  picture there are exactly two primes, so that $n=\ell_1^{\alpha_1}\ell_2^{\alpha_2}$).  Note that the transfer maps are uniquely determined by the associated restriction maps, since the pair is always $(1,\ell_i)$ or $(\ell_i,1)$.  

\begin{figure}[ht]
\begin{tikzpicture}[scale=0.9]
\draw (0,0)--(0,4)--(6,4)--(6,0)--(0,0);
\draw[dotted] (2,0) -- (2,4);
\draw[dotted] (4,0) -- (4,4);
\draw[dotted] (0,1)--(6,1);
\draw[dotted] (0,3)--(6,3);
\draw[->] (6,0.5)--(5.5,0.5);
\draw (5.75,0.3) node{$\scriptstyle{1}$};
\draw[->] (6,2)--(5.5,2);
\draw (5.75,1.8) node{$\scriptstyle{1}$};
\draw[->] (6,3.5)--(5.5,3.5);
\draw (5.75,3.3) node{$\scriptstyle{1}$};
\draw[->] (2,0.5)--(1.5,0.5);
\draw (1.75,0.3) node{$\scriptstyle{1}$};
\draw[->] (2,2)--(1.5,2);
\draw (1.75,1.8) node{$\scriptstyle{1}$};
\draw[->] (2,3.5)--(1.5,3.5);
\draw (1.75,3.3) node{$\scriptstyle{1}$};
\draw[->] (4,0.5)--(3.5,0.5);
\draw (3.75,0.3) node{$\scriptstyle{\ell_1}$};
\draw[->] (4,2)--(3.5,2);
\draw (3.75,1.8) node{$\scriptstyle{\ell_1}$};
\draw[->] (4,3.5)--(3.5,3.5);
\draw (3.75,3.3) node{$\scriptstyle{\ell_1}$};
\draw[->] (0.5,0)--(0.5,0.5);
\draw (0.7,0.2) node{$\scriptstyle{1}$};
\draw[->] (2.5,0)--(2.5,0.5);
\draw (2.7,0.2) node{$\scriptstyle{1}$};
\draw[->] (4.5,0)--(4.5,0.5);
\draw (4.7,0.2) node{$\scriptstyle{1}$};
\draw[->] (0.5,1)--(0.5,1.5);
\draw (0.7,1.2) node{$\scriptstyle{\ell_2}$};
\draw[->] (2.5,1)--(2.5,1.5);
\draw (2.7,1.2) node{$\scriptstyle{\ell_2}$};
\draw[->] (4.5,1)--(4.5,1.5);
\draw (4.7,1.2) node{$\scriptstyle{\ell_2}$};
\draw[->] (0.5,3)--(0.5,3.5);
\draw (0.7,3.2) node{$\scriptstyle{1}$};
\draw[->] (2.5,3)--(2.5,3.5);
\draw (2.7,3.2) node{$\scriptstyle{1}$};
\draw[->] (4.5,3)--(4.5,3.5);
\draw (4.7,3.2) node{$\scriptstyle{1}$};
\filldraw (2,3) circle (2pt);
\draw (2.2,2.8) node {$\scriptstyle{e}$};
\filldraw (4,1) circle (2pt);
\draw (3.8,1.2) node {$\scriptstyle{d}$};
\end{tikzpicture}
\label{fig:Z(d,e)}
\caption{The $\mZ$-module $\mZ(e;d)$ and its restriction maps.}
\end{figure}
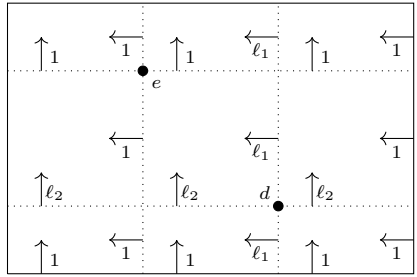

The portion of the picture that is above and left of the point $\Theta_d$ precisely matches the $\mZ$-module $I_e$.  A bit of thought shows that we have
a short exact sequence
\[ 0 \ra I_e \ra \mZ(e;d) \ra \mZ/I_d \ra 0.
\]
Also observe that $\mZ(e;1)\iso I_e$.  The following result is also easy from the definitions:

\begin{prop}
\label{pr:Z(d,e)-iso}
Let $b$ and $c$ be divisors of $n$.  Then there is an isomorphism $\mZ([b,c];b)\iso \mZ(c;(b,c))$.
\end{prop}

\begin{proof}
Write $b=\prod_i \ell_i^{b_i}$ and $c=\prod_i \ell_i^{c_i}$, so that $[b,c]=\prod_i \ell_i^{\max(b_i,c_i)}$ and $(b,c)=\prod_i \ell_i^{\min(b_i,c_i)}$.  In the definition of the restriction maps for $\mZ([b,c];b)$ the relevant intervals are $[b_i,\max(b_i,c_i))$, whereas for $\mZ(c;(b,c))$ the relevant intervals are $[\min(b_i,c_i),c_i)$.   If $c_i\geq b_i$ then these two intervals are the same, whereas if $c_i<b_i$ then both intervals are empty and hence also the same.  So $\mZ([b,c];b)$ and $\mZ(c;(b,c))$ have the same restriction maps, and hence also the same transfer maps. 
\end{proof}

\begin{remark}
\label{re:Z(e;d)-Q}
Here is another useful description of $\mZ(e;d)$. Recall for each orbit $\Theta_k$ we write $k=\prod_{i} \ell_i^{k_i}$. For each $i$ define the function
\[ \gamma_{[d_i,e_i)}(f)=\begin{cases}
0 & \text{if $f< d_i$}, \\
f-d_i & \text{if $d_i\leq f\leq e_i$}, \\
e_i-d_i & \text{if $f>e_i$}.
\end{cases}
\]
Then $\mZ(e;d)$ is
the sub-Mackey-functor of $\mQ$ whose value at $\Theta_k$ is the $\Z$-submodule generated by $\prod_i \ell_i^{-\gamma_i(k_i)}$.
\end{remark}

\begin{prop}
For all $d$ and $e$ where $d|e$, $\mZ(e;d)\tens \Q\iso \mQ$.  
\end{prop}

\begin{proof}
This is almost obvious from the description in Remark~\ref{re:Z(e;d)-Q}.  But it also 
follows directly from the exact sequences $0\ra I_e \ra \mZ(e;d)\ra \mZ/I_d\ra 0$ and $0\ra I_e\ra \mZ\ra \mZ/I_e\ra 0$.
\end{proof}

Finally, we give one more alternative construction of $\mZ(e;d)$.  This is in some ways the cleanest approach.  Consider the map $r\colon \mZ\oplus \mZ/I_d \ra \mZ/I_e$
that is the canonical projection on the first factor and on the second factor sends $1_{\pt} \mapsto \frac{e}{d}\cdot 1_{\pt}$. This second map is injective (use Proposition~\ref{pr:Z/F_a-props} and a little numerology) and so the kernel $\ker r$ is a form of $\Z$. The inclusion $I_e\inc \mZ$ factors through $\ker r$, and a diagram chase (or the Zig-Zag Lemma, if you like) using the diagram
\[ \xymatrix{
0\ar[r] & I_e \ar[r]\ar@{ >->}[d] & \mZ \ar[r]\ar@{ >->}[d] & \mZ/I_e\ar[r]
\ar[d]^{=} & 0\\
0 \ar[r] & \ker r \ar[r] & \mZ\oplus \mZ/I_d \ar[r] & \mZ/I_e \ar[r] & 0
}
\]
shows that the cokernel of $I_e\inc \ker r$ is $\mZ/I_d$.  With a little thought one can show directly that $\ker r\iso \mZ(e;d)$, or one can use Proposition~\ref{pr:type(e,d)-facts} below.

We next observe that $\mZ(e;d)$ is the unique solution to a certain extension problem.  

\begin{prop}
\label{pr:type(e,d)-facts}
Suppose $0\ra I_e \ra X \ra \mZ/I_d\ra 0$ is a short exact sequence where $X(\pt)$ is torsion-free.  Then $d|e$ and $X\iso \mZ(e;d)$.  Moreover, $\Ext^1(\mZ/I_d,X)=0$.  
\end{prop}

\begin{proof}
Say that an $X$ satisfying the conditions of the proposition is \mdfn{type $(e;d)$}.  Applying $\Hom(\mZ/I_d,\blank)$ to the short exact sequence gives
\[ 0 \ra \Z/d \ra \Z/(d,e) \ra \Ext^1(\mZ/I_d,X) \ra 0
\]
where we have used that $\Hom(\mZ/I_d,X)=0$ (because $X(\pt)$ is torsion-free), $\Ext^1(\mZ/I_d,I_e)\iso \Z/(d,e)$, and $\Ext^1(\mZ/I_d,\mZ/I_d)=0$.  For the left map in the exact sequence  to be an injection we must have $(d,e)=d$, or $d|e$.  But then the left map is an injection $\Z/d\ra \Z/d$ and thus an isomorphism. Hence $\Ext^1(\mZ/I_d,X)=0$.

Now let $Y$ be another module of type $(e;d)$.  Applying $\Hom(\blank,Y)$ to the sequence $0\ra I_e\ra X\ra \mZ/I_d\ra 0$ and using $\Ext^1(\mZ/I_d,Y)=0$
now shows that $\Hom(X,Y)\ra \Hom(I_e,Y)$ is an isomorphism. Similarly, $\Hom(X,X)\ra \Hom(I_e,X)$ is an isomorphism. 

From the above isomorphism statements (and their analogs with $X$ and $Y$ reversed) we obtain maps $f\colon X\ra Y$ and $g\colon Y\ra X$ making the following triangles commute:
\[ \xymatrix{
I_e \ar@{ >->}[r]^{i_X} \ar@{ >->}[d]_{i_Y} & X\ar[dl]^f
\\
Y 
}
\qquad
\xymatrix{
I_e \ar@{ >->}[r]^{i_X} \ar@{ >->}[d]_{i_Y} &  X
\\
Y. \ar[ur]_g
}
\]
Then $gfi_X=i_X$ and it follows from our isomorphism $\Hom(X,X)\ra \Hom(I_e,X)$ that $gf=\id_X$.  The analogous argument shows $fg=\id_Y$, and hence $X\iso Y$.

We have now shown that any two modules of type $(e;d)$ are isomorphic.  Since $\mZ(e;d)$ and $X$ are both such modules, $X\iso \mZ(e;d)$.  
\end{proof}

The following generalization is not immediately important, but it is good to have at our disposal for the future.  

\begin{cor}
Suppose the $\mZ$-module $X$ has a filtration
$X_0\subseteq X_1\subseteq \cdots \subseteq X_s=X$
with the properties that $X_0\iso I_e$ and $X_i/X_{i-1}\iso \mZ/I_{d_i}$ for all $i\geq 1$.  Also assume that $X(\pt)$ is torsion-free.  Then $d_i|e$ for all $i$, $(d_i,d_j)=1$ whenever $i\neq j$, and $X\iso \mZ(e;\prod_i d_i)$ (note that $\prod_i d_i$ divides $e$).  
\end{cor}

\begin{proof}
First observe that $X/X_0$ has the filtration $\{X_i/X_0\}_i$ whose quotients are the $\mZ/I_{d_i}$.  Since $\Ext^1(\mZ/I_b,\mZ/I_c)=0$ for any $b$ and $c$, it follows that $X/X_0\iso \bigoplus_i \mZ/I_{d_i}$.  We can filter this module so that the quotients are the $\mZ/I_{d_i}$ in any order, and so the same applies to $X$.  In other words, we are free to change the original filtration in order to permute the $d_i$'s.  

The module $X_1$ is of the type considered in Proposition~\ref{pr:type(e,d)-facts} and so we conclude $d_1|e$ and $\Ext^1(\mZ/I_{d_1},X_1)=0$.  

Applying $\Hom(\mZ/I_{d_1},\blank)$ to $0\ra X_1 \ra X_2 \ra \mZ/I_{d_2}\ra 0$ gives the exact sequence 
\[ 0=\Hom(\mZ/I_{d_1},X_2)\ra \Z/(d_1,d_2) \ra \Ext^1(\mZ/I_{d_1},X_1)=0.
\]
Note that the first $\Hom$ is zero because $X_2(\pt)$ is torsion-free.  It follows that $\Z/(d_1,d_2)=0$, and so $(d_1,d_2)=1$.  

By permuting the $d_i$ we now conclude that $d_i|e$ for all $i$, and that $(d_i,d_j)=1$ for any $i\neq j$.

Now return to the first paragraph, where we showed that we have the short exact sequence 
\[ 0 \ra I_e=X_0 \ra X \ra \bigoplus_i \mZ/I_{d_i} \ra 0.
\]
Since the $d_i$ are pairwise relatively prime, Remark~\ref{re:decompose} implies that the module on the right is isomorphic to $\mZ/I_{D}$ where $D=\prod_i d_i$.  So by Proposition~\ref{pr:type(e,d)-facts} we conclude $X\iso \mZ(e;\prod_i d_i)$.  
\end{proof}

\begin{remark}
Do not mistake the $\mZ(e;d)$ modules as nicely-behaved objects, simply because their values are all equal to $\Z$.    For example, when $n=\ell^2$ one can prove that 
\[ \mZ(\ell^2;\ell)\bbox\, \mZ(\ell^2;\ell)\iso \mZ/I_\ell \oplus \mZ(\ell^2;1).
\]
The appearance of the torsion term is not at all intuitive here.
\end{remark}

\section{Applications to equivariant topology}
\label{se:appl}

In this section we give a very brief tour of some applications to Bredon cohomology calculations.  Readers are referred to \cite{DH2} for a much deeper treatment.

Fix $G=C_n$.  For each $d|n$ let $S^{\lambda_d}$ be the chain complex
\[ 0 \lra F_d \llra{\id-Rt} F_d \llra{R\pi} \mZ \lra 0.
\]
This arises as the cellular chain complex for a certain evident cell structure on the one-point compactification of $\C$ where $C_n$ has $t$ acting as multiplication by $e^{\tfrac{2\pi i d}{n}}$, but this geometric input can here be taken as a black box.  Note that $S^{\lambda_d}$ is a truncation of the standard free resolution of $\mZ/I_d$, so we can read off that
\[ \uH_*(S^{\lambda_d})\iso \begin{cases}
\mZ & \text{if $i=2$}, \\
0 & \text{if $i=1$ or $i\geq 3$}, \\
\mZ/I_d & \text{if $i=0$.}
\end{cases}
\]
Computing the regular portion of the Bredon cohomology groups $H^\star_G(\pt)$ amounts to computing
\[ \uH_*(S^{\lambda_{d_1}}\bbox \cdots \bbox S^{\lambda_{d_k}}) \qquad\text{and}\qquad
\uH_*\bigl (\uHom(S^{\lambda_{d_1}}\bbox \cdots \bbox S^{\lambda_{d_k}},\mZ)\bigr )
\]
for any collection of divisors $d_1,\ldots,d_k$ of $n$.  
Of course this is a completely algebraic problem about chain complexes of $\mZ$-modules, and one can attempt to tackle it with the standard tools of homological algebra.

The K\"unneth and Universal Coefficient spectral sequences work for $\mZ$-modules just as they do in any reasonable abelian category.  If $C_\bullet$ is a bounded-below complex of free $\mZ$-modules and $D_\bullet$ is any chain complex of $\mZ$-modules, then there is a K\"unneth spectral sequence
\[ E^2_{p,q}=\bigoplus_{q_1+q_2=q} \uTor_p(\uH_{q_1}(C),\uH_{q_2}(D))\Rightarrow \uH_{p+q}(C\bbox D).
\]
Here $E^2_{p,q}$ is drawn on a grid in the $(p,q)$ spot, and the differentials take the form $d_r\colon E^r_{p,q}\lra E^r_{p-r,q+r-1}$.  The Universal Coefficient spectral sequence (perhaps better described as the spectral sequence for a  mapping space) has
\[ E_2^{p,q}=\bigoplus_{q_2-q_1=q} \uExt^{-p}(\uH_{q_1}(C),\uH_{q_2}(D)) \Rightarrow \uH_{p+q}(\uHom(C,D)).
\]
Here $E_2^{p,q}$ is again drawn in the $(p,q)$ spot (and so is concentrated in the half-plane $p\leq 0$) and the differentials again take the form $d_r\colon E_r^{p,q}\ra E_r^{p-r,q+r-1}$.  

See also \cite[Section 4.4]{Z} for these spectral sequences.

\begin{example}
Let us attempt to compute $\uH_*(S^{\lambda_b}\bbox S^{\lambda_c})$ using the K\"unneth spectral sequence.  The $E^2$-term is given in Figure~\ref{fig:sample1}.  There is a single possible nonzero differential, the indicated $d^3$.  It turns out that this differential exists and is the evident injection that we encountered back in Proposition~\ref{pr:mZ-ses}.  The argument for this is not transparent, and we refer the reader to the methods of \cite{DH2} (which avoid the spectral sequence, but nevertheless lead to this conclusion).  As a result one obtains

\[ \uH_i(S^{\lambda_b}\bbox S^{\lambda_c}) \iso \begin{cases}
\mZ & \text{if $i=4$},\\
\mZ/I_{[a,b]} & \text{if $i=2$}, \\
\mZ/I_{(a,b)} & \text{if $i=0$},\\
0 & \text{otherwise.}
\end{cases}
\]

\begin{figure}[ht]
\begin{tikzpicture}[scale=0.5]
\draw (0.2,-0.2) node{$\mZ/I_{(a,b)}$};
\draw (12,-0.2) node{$\mZ/I_{(a,b)}$};
\draw (0,2.8) node{$\mZ/I_a\oplus \mZ/I_b$};
\draw (0,5.4) node{$\mZ$};
\draw[->,thick] (10.8,0.1) -- (0.1,2.4);
\draw (7,1.5) node{??};

\foreach \y in {-1,0.5,2,3.5,5,6.5}
	{\draw (-2,\y)--(14,\y);}
\foreach \x in {-2,2,6,10,14}
    {\draw (\x,-1) -- (\x,6.5);}
\filldraw (-2,-1) circle(.1);
\draw[->,thick] (-2,-1)--(14.2,-1);
\draw[->,thick] (-2,-1)--(-2,6.7);
\draw (-2.4,6.6) node{$\scriptstyle{q}$};
\draw(14.4,-1) node{$\scriptstyle{p}$};
\end{tikzpicture}
\label{fig:sample1}
\caption{K\"unneth spectral sequence for $\uH_*(S^{\lambda_a}\bbox S^{\lambda_b})$}
\end{figure}
\end{example}

\begin{example}
Next let us try to compute the homology modules of $\uHom(S^{\lambda_a},S^{\lambda_b})$.  We can attempt this directly using the Universal Coefficient spectral sequence, or we can use the isomorphism
\[ \uHom(S^{\lambda_a},S^{\lambda_b})\iso \uHom(S^{\lambda_a},\mZ)\bbox S^{\lambda_b}
\]
(which follows from Corollary~\ref{co:box-hom})
and then use the K\"unneth spectral sequence.  We will attempt it both ways. 

The $E_2$-term for the Universal Coefficient spectral sequence is in Figure~\ref{fig:sample2}.  This time there are two $d_3$'s to analyze, and neither one is easy to evaluate.

\begin{figure}[ht]
\begin{tikzpicture}[scale=0.5]
\draw (0.2,2.8) node{$\mZ/I_{(a,b)}$};
\draw (12,-0.2) node{$\mZ/I_{b}$};
\draw (12,2.8) node{$\mZ/I_{(a,b)}\oplus \mZ$};
\draw (0,5.6) node{$\mZ/I_a$};
\draw[->,thick] (10.8,0.1) -- (0.1,2.4);
\draw (7,1.5) node{??};
\draw[->,thick] (10.1,3.2) -- (0.8,5.4);
\draw (7,4.5) node{??};

\foreach \y in {-1,0.5,2,3.5,5,6.5}
	{\draw (-2,\y)--(14,\y);}
\foreach \x in {-2,2,6,10,14}
    {\draw (\x,-1) -- (\x,6.5);}
\filldraw (14.0,2) circle(.1);
\draw[->,thick] (14.0,2)--(-2.3,2);
\draw[->,thick] (14.0,-1)--(14.0,6.7);
\draw[->,thick] (14,2)--(14.6,2);
\draw (14.5,6.4) node{$\scriptstyle{q}$};
\draw(14.9,2) node{$\scriptstyle{p}$};
\end{tikzpicture}
\label{fig:sample2}
\caption{Universal Coefficient spectral sequence for $\uH_*(\uHom(S^{\lambda_a}, S^{\lambda_b}))$}
\end{figure}

Let us instead look at the K\"unneth spectral sequence for $\uHom(S^{\lambda_a},\mZ)\bbox S^{\lambda_b}$.  By inspection $\uHom(S^{\lambda_a},\mZ)$ is a truncation of our standard resolution for $\mZ/I_a$, and so its homology modules are easy to calculate: they are $I_a$ in degree $-2$ and zero elsewhere.  The $E^2$-term is  then readily computed and is in Figure~\ref{fig:sample3}.  This time there is no room for differentials, and the spectral sequence collapses.  We do have an extension problem for $H_0$, though.  It turns out that $H_0$ is torsion-free and so by Proposition~\ref{pr:type(e,d)-facts} the module in question is $\mZ(a;(a,b))$.  We again refer to \cite{DH2} for an approach to some of the missing details here.  (Note that by comparing the two spectral sequences we now see that both $d_3$'s in the Universal Coefficient spectral sequence must be surjective).  

\begin{figure}[ht]
\begin{tikzpicture}[scale=0.5]
\draw (0.2,-0.2) node{$I_{(a,b)}/I_b$};
\draw (8,-0.2) node{$\mZ/I_{(a,b)}$};
\draw (0,2.8) node{$I_a$};
\foreach \y in {-1,0.5,2,3.5,5}
	{\draw (-2,\y)--(14,\y);}
\foreach \x in {-2,2,6,10,14}
    {\draw (\x,-1) -- (\x,5);}
\filldraw (-2,2) circle(.1);
\draw[->,thick] (-2,2)--(14.3,2);
\draw[->,thick] (-2,-1)--(-2,5.4);
\draw (-2.4,5.6) node{$\scriptstyle{q}$};
\draw(14.5,2) node{$\scriptstyle{p}$};
\end{tikzpicture}
\caption{K\"unneth spectral sequence for $\uH_*\bigl (\uHom(S^{\lambda_a},\mZ) \bbox S^{\lambda_b}\bigr )$}
\label{fig:sample3}
\end{figure}

If we were to now try to compute the homology of $\bigl( \uHom(S^{\lambda_a},\mZ)\bbox S^{\lambda_b}\bigr )\bbox S^{\lambda_c}$ we would need to use some of the $\uTor$  modules from Proposition~\ref{pr:Ia/Ib}, because of the $I_{(a,b)}/I_b$ terms from the previous computation.  But clearly things get a bit hairy at this point.  
\end{example}

Again, the content of this section is mostly just to demonstrate the kinds of settings where the earlier work in the paper proves quite useful.  For more in this direction, see \cite{DH2}.


\section{Change of groups}
\label{se:change}
Suppose one wants to understand $\uTor_*^{\mZ_{C_n}}(\mZ/I_a,\mZ/I_b)$.  As we know from Remark~\ref{re:decompose}, there is a decomposition $\mZ/I_a\iso \bigoplus_{\ell} \mZ/I_{a(\ell)}$ where the sum is over all primes $\ell$, and $a(\ell)$ is the $\ell$-adic part of $a$.  The analogous decomposition of $\mZ/I_b$ allows one to immediately reduce to the case of $\uTor_*^{\mZ_{C_n}}(\mZ/I_{\ell^e},\mZ/I_{\ell^f})$.  The goal of this section is to explain a further reduction where one changes $C_n$ to $C_{n(\ell)}$.  For this we need to review the general tools for changing the base group.

\subsection{Restriction and induction}
\label{se:res-ind}
Let us first recall some general category theory.  If $I$ and $J$ are $\Ab$-categories (categories enriched over abelian groups) say that a functor $\alpha\colon I\ra J$ is additive if the maps on hom-sets are additive.  Write $\Ab^I$ for the category of additive functors, and write $\Pre(\alpha)\colon \Ab^J\ra \Ab^I$ for precomposition with $\alpha$.  The functor $\Pre(\alpha)$ has left and right adjoints, called additive Kan extensions and denoted $L_\alpha$ and $R_\alpha$.
For an additive functor $M\colon I\ra \Ab$ these are 
given by the formulas
\[ (L_\alpha M)(j)=\coeq \Biggl [ 
\bigoplus_{i\ra i'} J(\alpha i',j)\tens Mi
\dbra
\bigoplus_i J(\alpha i,j)\tens Mi \Biggr ]
\]
(the two maps are the evident ones) and
\[ (R_\alpha M)(j)=\eq \Biggl [
\prod_i (Mi)^{J(j,\alpha i)} \dbra
\prod_{i\ra i'} (Mi')^{J(j,\alpha i)}
\Biggr ].
\]
In the latter formula we use the notation $A^B=\Hom(B,A)$ for abelian groups $A$ and $B$.  

Suppose we are given an additive functor $\beta\colon J\ra I$.  An easy exercise shows that
the structure of an adjunction $(\alpha,\beta)$ (with $\alpha$ the left adjoint) determines a natural isomorphism $L_\alpha M\ra M\beta$, and the structure of an adjunction $(\gamma,\alpha)$ determines a natural isomorphism $M\gamma \ra R_\alpha M$.  (Essentially, one uses the unit and counit of the adjunction to write down augmentation maps for the equalizer/coequalizer together with a contracting homotopy.)  Both of these can be rephrased as saying that when $(\alpha,\beta)$ is an adjunction then $(\Pre(\beta),\Pre(\alpha))$ is an adjunction.

Now suppose that $f\colon G_1\ra G_2$ is a homomorphism of finite groups.  We have the usual adjoint pairs
\[ \xymatrix{ 
\Z[G_1]\MMod \ar@/^6ex/[r]^{\Ind_f} \ar@/_5ex/[r]_{\coInd_f} \ar @/^4ex/ @{{}{ }{}} [r]|{\perp} 
\ar @/_2.5ex/ @{{}{ }{}} [r]|{\perp}
& \Z[G_2]\MMod \ar[l]_-{\Res_f} 
}
\]
where the induction and coinduction functors are $\Ind_f(M)=\Z[G_2]\tens_{\Z[G_1]} M$ and $\coInd_f(M)=\Hom_{\Z[G_1]}(\Z[G_2],M)$.  Recall that when $f$ is an inclusion these two functors are naturally isomorphic, by a standard argument.

The restriction and induction functors applied to a permutation module yield, canonically, another permutation module, so these functors yield a corresponding adjoint pair 
\begin{equation}
\label{eq:BZ-adjoint}
\xymatrix{
B\Z_{G_1} \ar@/^5ex/[r]^{\Ind_f}  \ar@/^3ex/@{{}{ }{}}[r]|{\perp}& B\Z_{G_2}. \ar[l]^{\Res_f}
}
\end{equation}
When $f$ is an inclusion this is also true of $\coInd_f$, because of the isomorphism with $\Ind_f$.  But for a projection $G\ra G/N$ and a $G$-set $S$, there is no isomorphism between $\Z\langle S\rangle^N$ and $\Z\langle S/N\rangle$ that is canonical in $S$ and so coinduction does not yield a functor between the Burnside categories.

Upon taking opposite categories the adjoint pair of (\ref{eq:BZ-adjoint}) yields a similar adjunction but with $\Res_f$ now the left adjoint.  Then passing to diagram categories we have
\begin{equation}
\label{eq:BZ-adjoint2}
\xymatrix{
\mZ_{G_1}\MMod 
\ar@/^3ex/@{{}{ }{}}[r]|{\perp}
\ar@/_4ex/@{{}{ }{}}[r]|{\perp}
\ar[r]_{\Pre(\Res_f)}
& \mZ_{G_2}\MMod. 
\ar@/^7ex/[l]^{R_{\Res_f}}
\ar@/_5ex/[l]_{\Pre(\Ind_f)} 
}
\end{equation}
When $f$ is an inclusion $R_{\Res_f}=\Pr(\Ind_f)$.  

It is common to define induction and restriction of Mackey functors by
\[ \mInd_f=\Pr(\Res_f) \quad\text{and}\quad \mRes_f=\Pr(\Ind_f)
\]
(see, for example, \cite[Section 3]{W}).  This can be a little confusing when one needs to go back to the perspective of Kan extensions; for example, the left Kan extension along the restriction functor $\Res_f$ is again ``restriction'', but here in the guise of $\mRes$.  Unfortunately, this is a case where no choice of terminology seems to be entirely satisfying.

Note that restriction $\mRes$ is always the left adjoint to induction $\mInd$ (again, this is slightly awkward), and when $f$ is an inclusion then it is also the right adjoint to induction.  When $f$ is an inclusion   
it is common to write $\mInd_f(M)=M\ind_{G_1}^{G_2}$ and $\mRes_f(M)=M\res^{G_2}_{G_1}$.  Let us record the formulas
\[ (\mRes_f M)(\!\phantom{|}_{G_1} S)=M(G_2\times_{G_1} S), \qquad
(\mInd_f N)(\!\phantom{|}_{G_2} T)=N(\!\phantom{|}_{G_1}T).
\]
Note that these formulas immediately imply that $\mRes_f$ and $\mInd_f$ are both exact functors, since for chain complexes of Mackey functors exactness is determined objectwise.  

\begin{remark}
It is useful to have a rough image of what these functors do that relates back to pictures of the orbit category.  
When $f$ is an inclusion, $\mRes_f(M)$ takes the values of $M$ on the ``large'' part of the orbit diagram, from $G_2$ to $G_2/G_1$, relabelling the entries to match the orbit category of $G_1$.  The induction functor is more subtle to describe in terms of a simple picture.

When $f$ is a projection $G\ra G/N$ ($N$ a normal subgroup of $G$) and $M$ is a $\mZ_G$-module, $\mInd_f(M)$ just takes the values of $M$ on the ``small'' part of the orbit diagram, from $G/N$ down to $G/G$.  If $X$ is a $\mZ_{G/N}$-module then $\mRes_f X$ is precisely $X$ on these ``small'' orbits, but then is freely extended into the rest.\end{remark}

The following results follow readily from the adjunctions and definitions:

\begin{prop} 
\label{pr:group-change}
Let $f\colon G_1\ra G_2$ be a homomorphism.
\begin{enumerate}[(a)]
\item If $T$ is a $G_2$-set then $\mRes_f (F_{T})\iso F_{({}_{G_1}\!T)}$ naturally in $T$.  
\item If $f$ is an inclusion and $S$ is a $G_1$-set then $\mInd_f(F_S)\iso F_{(G_2\times_{G_1}S)}$ naturally in $S$.  
\item For the projection $f\colon G\ra G/N$, where $N$ is a normal subgroup, one has $\mInd_f\circ \mRes_f \iso \id$.  
\end{enumerate}
\end{prop}

\begin{proof}
Left to the reader.
\end{proof}

\begin{remark}
\label{re:res-free}
Note in particular that Proposition~\ref{pr:group-change}(a) yields $\mRes_f \mZ_{G_2}=\mRes_f F_{\pt}=F_{(\!\phantom{|}_{G_1}\pt)}=\mZ_{G_1}$. Additionally, when $f$ is a surjection and $H\supseteq \ker f$ it gives that $\mRes_f F_{G_1/H}=F_{G_1/H}$, where the two free modules are slightly different: the first is a free $\mZ_{G_2}$-module and the second is a free $\mZ_{G_1}$-module.   
\end{remark}

Returning to the general case of $f\colon G_1\ra G_2$, the commutative diagram
\[ \xymatrix{
B\Z_{G_2}^{op} \times B\Z_{G_2}^{op} \ar[d]_{\Res_f\times \Res_f}\ar[r]^-{\times} & B\Z_{G_2}^{op} \ar[d]^{\Res_f} \\
B\Z_{G_1}^{op} \times B\Z_{G_1}^{op} \ar[r]^-{\times} & B\Z_{G_1}^{op}  \\
}
\]
yields a natural isomorphism between successive left Kan extensions corresponding to the two ways of going around the diagram.  This immediately yields the following:

\begin{prop}
\label{pr:Res-monoidal}
In the above setting there is a natural isomorphism
\[ \mRes_f(M\bbox_{G_2} N) \iso (\mRes_f M)\bbox_{G_1} (\mRes_f N).
\]
\end{prop}

We can now give a very short proof for a result we mentioned back in Section~\ref{se:background}.

\begin{proof}[Proof of Proposition~\ref{pr:box-on-G}]
Suppose $f$ is the inclusion $\{e\}\inc G$.  If $M$ is a $\mZ_G$-module then $\mRes_f(M)=M(G)$.  So Proposition~\ref{pr:Res-monoidal} gives a natural isomorphism $(M\bbox N)(G)\iso M(G)\tens N(G)$.  A little legwork shows this to be the same as the map introduced in Section~\ref{se:background}.
\end{proof}

We also have the following change-of-groups theorem for $\uTor$:

\begin{prop}
\label{pr:res-Tor}
Let $f\colon G_1\ra G_2$ be a homomorphism, and let $M$ and $N$ be $\mZ_{G_2}$-modules.  Then
\[ \uTor_*^{\mZ_{G_1}}(\mRes_f M,\mRes_f N)\iso \mRes_f \Bigl [\uTor_*^{\mZ_{G_2}}(M,N) \Bigr ].
\]
\end{prop}

\begin{proof}
Let $J_\bullet\ra M\ra 0$ be a free resolution over $\mZ_{G_2}$.  Then $\mRes_f J_\bullet \ra \mRes_f M\ra 0$ is a free resolution over $\mZ_{G_1}$, using Proposition~\ref{pr:group-change}(a) and the fact that $\mRes_f$ is exact.  But
\[ (\mRes_f J_\bullet)\bbox (\mRes_f N)\iso \mRes_f(J_\bullet \bbox N) \]
by Proposition~\ref{pr:Res-monoidal}, and since $\mRes_f$ is exact it commutes with taking  homology.  
\end{proof}

\begin{remark}
Although we won't use it, we mention the following change-of-groups result for $\uExt$.  Let $f\colon G_1\ra G_2$, let $M$ be a $\mZ_{G_2}$-module, and let $N$ be a $\mZ_{G_1}$-module.  Then
\[ \mInd_f \Bigl[ \uExt^i_{\mZ_{G_1}}(\mRes_f M,N) \Bigr ]\iso \uExt^i_{\mZ_{G_2}}(M,\mInd_f N).
\]
The result for $\uHom$ can be checked directly from the definitions, and then the extension to $\uExt$ is almost immediate using the same ideas as in Proposition~\ref{pr:res-Tor}.
\end{remark}

\subsection{Change of groups and families}

Let $f\colon G_1\ra G_2$ be a homomorphism and let $\cF$ be a family of subgroups of $G_2$.  Define $f^*\cF=\{K\leq G_1\,|\, f(K)\in \cF\}$, and observe that this is a family of subgroups of $G_1$.  Note that $f^*\cF$ contains $\{f^{-1}(H)\,|\, H\in \cF\}$, and in fact is obtained from this set by adding in all subgroups of its elements.

The surjection $\mZ_{G_2}\ra \mZ_{G_2}/\cF$ yields a canonical map
\[ \mZ_{G_1}=\mRes_f \mZ_{G_2} \fib \mRes_f (\mZ_{G_2}/\cF).
\]
For $K\in f^*\cF$ we have
$[\mRes_f (\mZ_{G_2}/\cF)](G_1/K)=[\mZ_{G_2}/\cF](G_2/f(K)) =0$,
the vanishing because $f(K)\in \cF$.  So our canonical map descends to give 
\[ \mZ_{G_1}/f^*\cF \fib \mRes_f (\mZ_{G_2}/\cF).
\]

\begin{prop}
\label{pr:res-Z/F}
When $f\colon G_1\ra G_2$ is surjective, the above map $\mZ_{G_1}/f^*\cF \ra \mRes_f(\mZ_{G_2}/\cF)$ is an isomorphism.
\end{prop}

\begin{proof}
We have the exact sequence $\bigoplus_{H\in \cF} F_{G_2/H} \ra \mZ_{G_2} \ra \mZ_{G_2}/\cF \ra 0$ of $\mZ_{G_2}$-modules.  Applying $\mRes_f$ yields a similar exact sequence of $\mZ_{G_1}$-modules of the form
\[ \bigoplus_{H\in \cF} F_{G_1/f^{-1}(H)}\ra \mZ_{G_1} \ra \mRes_f(\mZ_{G_2}/\cF)\ra 0
\]
(note we have used Remark~\ref{re:res-free} here and that surjectivity of $f$ implies $G_2/H$ is a transitive $G_1$-set given by $G_1/f^{-1}(H)$).  This identifies $\mRes_f(\mZ_{G_2}/\cF)$ with $\mZ_{G_1}/J$ where $J\subseteq \mZ_{G_1}$ is the ideal generated by the elements $1_{G_1/f^{-1}(H)}$ for $H\in \cF$.  

In contrast, $I_{f^*\cF}$ is the ideal generated by elements $1_{G_1/K}$ for $f(K)\in \cF$.  It is trivial that $J\subseteq I_{f^*\cF}$, but in fact they are equal: this is because $K\subseteq f^{-1}(f(K))$ and $1_{G_1/K}=p^*(1_{G_1/f^{-1}(f(K))})$.
\end{proof}

\begin{cor}
\label{co:tor-change}
Let $f\colon G_1\ra G_2$ be a surjection and let $\cF_a$ and $\cF_b$ be families of subgroups of $G_2$.  Then
\[ \uTor_*^{\mZ_{G_1}}\Bigl (\mZ_{G_1}/f^*\cF_a,\ \mZ_{G_1}/f^*\cF_b\Bigr )\iso \mRes_f \Bigl [ \uTor_*^{\mZ_{G_2}}(\mZ_{G_2}/\cF_a,\ \mZ_{G_2}/\cF_b)\Bigr].
\]
\end{cor}

\begin{proof}
Immediate from Propositions~\ref{pr:res-Z/F} and \ref{pr:res-Tor}.
\end{proof}

\subsection{Applications}

Fix a prime $\ell$.  Recall that $n(\ell)$ is the $\ell$-primary part of $n$,  and $n(\hat{\ell})=\frac{n}{n(\ell)}$.
We have the extension of groups
\[ 0 \ra C_{n(\hat{\ell})} \inc C_n \llra{f} C_{n(\ell)}  \ra 0.
\]
By Proposition~\ref{pr:group-change}(a) we have 
$\mRes_f(F_{\Theta_{\ell^e}})=F_{\Theta_{\ell^e}}$.
Note that the two free functors, though written the same, are slightly different: the first is a $\mZ_{C_{n(\ell)}}$-module and the second is a $\mZ_{C_n}$-module.  

Let $d=\ell^e$.  
Recall that the $C_{n(\ell)}$-family $\cF^{C_{n(\ell)}}_{d}$ consists of all subgroups of $C_{n(\ell)}$ whose order divides $\tfrac{n(\ell)}{d}$.  Since $|f(K)|=|K|(\ell)$, we have that $f^*(\cF^{C_{n(\ell)}}_d)$
consists of all subgroups $K\leq C_n$ such that $|K|(\ell)$ divides $\frac{n(\ell)}{d}$.  But note that this is equivalent to the condition that $|K|$ divides $\frac{n}{d}$, which shows that
\[ f^*(\cF^{C_{n(\ell)}}_d)=\cF^{C_n}_d.
\]
Therefore
Proposition~\ref{pr:res-Z/F} yields that $\mRes_f (\mZ_{C_{n(\ell)}}/I_d)=\mZ_{C_n}/I_d$.  Finally,
Corollary~\ref{co:tor-change} now gives us that
\[ \uTor_*^{\mZ_{C_n}}(\mZ_{C_n}/I_{\ell^e},\ \mZ_{C_n}/I_{\ell^f})
\iso \mRes_f \Bigl [ \uTor_*^{\mZ_{C_{n(\ell)}}}(\mZ_{C_{n(\ell)}}/I_{\ell^e},\ \mZ_{C_{n(\ell)}}/I_{\ell^f} )\Bigr ].
\]
This is the maneuver that allows us to connect the $\uTor$-modules over $C_n$ and $C_{n(\ell)}$.  

We give one more application, which shows how the results of this paper can be extended beyond the setting of cyclic groups.  There are exactly two nontrivial families for the symmetric group $\Sigma_3$, namely $\cF_{[3]}=\{e,C_3\}$ and $\cF_{[ 2]}=\{e,\langle (12)\rangle, \langle(13)\rangle, \langle (23)\rangle\}$.  If $f\colon \Sigma_3\ra C_2$ is the (unique) quotient map, then $\cF_{[3]}=f^*\cF_2$.  So Corollary~\ref{co:tor-change} gives that
\[ \uTor_*^{\mZ_{\Sigma_3}}(\mZ/\cF_{[3]},\mZ/\cF_{[3]})\iso \mRes_f \Bigl [\uTor_*^{\mZ_{C_2}}(\mZ/I_2,\mZ/I_2)\Bigr ].
\]
Of course we know the right hand side, by our work in previous sections.  


\appendix
\section{A leftover proof}
\label{se:appendix}

This section deals with the proof of
Proposition~\ref{pr:box-slice-isos}, which was skipped in the body of the text. We recall the proposition here for convenience: \medskip

\hypersetup{linkcolor=black}
\noindent \textbf{Proposition~\ref{pr:box-slice-isos}.}
\hypersetup{linkcolor=red}
\emph{There is a natural isomorphism $F_A\bbox M\iso M^{DA}$ and also a natural isomorphism $\uHom(F_A,M) \iso M^A$.  Consequently, note the natural isomorphism
\[ \uHom(F_A,\mZ)\iso \mZ^A \iso F_{DA}.
\]}\medskip

 We begin by producing the natural map $F_A\bbox M\ra M^{DA}$.  
Recall that for any finite $G$-set $X$, $(F_A\bbox M)(X)$ is the coequalizer of the diagram
\begin{center}
\begin{tikzcd}
  \coprod\limits_{\smash{C_1,C_2,D_1,D_2}}
  \cBZ(X,C_1\times C_2) \tens \cBZ(C_1,D_1)\tens \cBZ(C_2,D_2) \tens F_A(D_1)\tens M(D_2) \ar[d, shift left] \ar[d, shift right]
\\
\coprod\limits_{C_1,C_2} 
\cBZ(X,C_1\times C_2) \tens F_A(C_1)\tens M(C_2)
\end{tikzcd}
\end{center}
\noindent
which we rewrite as
\begin{center}
\begin{tikzcd}
\!\!\!\!  \coprod\limits_{\smash{C_1,C_2,D_1,D_2}}
  \cBZ(X,C_1\times C_2) \tens \cBZ(C_1,D_1)\tens \cBZ(C_2,D_2) \tens \cBZ(D_1,A)\tens M(D_2) \ar[d, shift left] \ar[d, shift right]
\\
\coprod\limits_{C_1,C_2} 
\cBZ(X,C_1\times C_2) \tens \cBZ(C_1,A)\tens M(C_2).
\end{tikzcd}
\end{center}
\noindent
Given $T\in \cBZ(X,C_1\times C_2)$ and $S\in \cBZ(C_1,A)$ we can form $(S\times \id_A)\circ T\in \cBZ(X,A\times C_2)$ and then use adjointness to obtain a corresponding element in $\cBZ(DA\times X,C_2)$.  Pullback along this map allows us to pass from $M(C_2)$ to $M(DA\times X)=M^{DA}(X)$.  Said differently, we are taking the composition
\[ \xymatrixrowsep{1.5pc}\xymatrix{
\cBZ(X,C_1\times C_2) \tens \cBZ(C_1,A)\tens M(C_2)\ar[d] \\
\cBZ(X,C_1\times C_2)\tens \cB_Z(C_1\times C_2,A\times C_2)\tens M(C_2)\ar[d]\\
\cBZ(X,A\times C_2)\tens M(C_2) \ar[d] \\
\cBZ(DA\times X,C_2)\tens M(C_2)\ar[d]\\
M(DA\times X)=M^{DA}(X).
}
\]
This composition is readily checked to respect the coequalizer, and therefore induces $(F_A\bbox M)(X)\ra M^{DA}(X)$, that is natural in both $X$ and $A$.  So we have a map of $\mZ$-modules $\sigma\colon F_A\bbox M\ra M^{DA}$ that is natural in $A$.  

\begin{prop} 
\label{pr:box-isos}
\mbox{}\par
\begin{enumerate}[(a)]
\item The map $\sigma$ is an isomorphism when $M$ is a free module $F_B$.  
\item The map $\sigma$ is an isomorphism for every $\mZ$-module $M$.  
\end{enumerate}
\end{prop}

\begin{proof}
Part (b) follows immediately from (a), since every $\mZ$-module is a colimit of free modules and both $F_A\bbox(\blank)$ and $(\blank)^{DA}$ preserve colimits.  
So we only need to prove (a).

Recall again that $(F_A\bbox F_B)(X)$ is the coequalizer of 
\begin{center}
\begin{tikzcd}  
\!\!\!\!\!\coprod\limits_{{C_1,C_2,D_1,D_2}}\!\!\!\!\!\!\!\!
\cBZ(X,C_1\times C_2) \tens \cBZ(C_1,D_1)\tens \cBZ(C_2,D_2) \tens \cBZ(D_1,A)\tens \cBZ(D_2,B) \arrow[d, shift left] \ar[d, shift right]
\\
\coprod\limits_{C_1,C_2} 
\cBZ(X,C_1\times C_2) \tens \cBZ(C_1,A)\tens \cBZ(C_2,B).
\end{tikzcd}
\end{center}
Cartesian product and composition yields an evident map $\beta$ to $\cBZ(X,A\times B)$ that is compatible with the coequalizer maps.
This is not identical to the map $\sigma$, but they are easily seen to be related by the diagram below:

\[ \xymatrix{
(F_A\bbox F_B)(X) \ar[rr]^-\beta \ar[d]_{\sigma} && \cBZ(X,A\times B) \\
(F_B)^{DA}(X) \ar@{=}[r] & F_B(DA\times X)\ar@{=}[r] &\cBZ(DA\times X,B)\ar[u]_\iso
}
\]
(the right vertical map is our familiar adjunction isomorphism).  So we can prove that $\sigma$ is an isomorphism by instead proving that $\beta$ is an isomorphism.  However, this is almost trivial: the coequalizer diagram below has an obvious pair of contracting homotopies
\begin{center}
\begin{tikzcd}  
\!\!\!\!\!\coprod\limits_{{C_1,C_2,D_1,D_2}}\!\!\!\!\!\!\!\!
\cBZ(X,C_1\times C_2) \tens \cBZ(C_1,D_1)\tens \cBZ(C_2,D_2) \tens \cBZ(D_1,A)\tens \cBZ(D_2,B) \arrow[d, shift left] \ar[d, shift right]
\\
\coprod\limits_{C_1,C_2} 
\cBZ(X,C_1\times C_2) \tens \cBZ(C_1,A)\tens \cBZ(C_2,B)\arrow[d] \arrow[u, bend right=40,"S'" right]\\
\cBZ(X,A\times B) \arrow[u, bend right=40,"S" right]
\end{tikzcd}
\end{center}
where $S$ maps into the $C_1=A$, $C_2=B$ summand in the evident way, and $S'$ maps into the $D_1=A$, $D_2=B$ summand similarly.  
\end{proof}

It was established in the course of the above proof that there is 
a natural isomorphism $F_A\bbox F_B\iso F_{A\times B}$.
We will use this to complete our goal:

\begin{proof}[Proof of Proposition~\ref{pr:box-slice-isos}]
The isomorphism $F_A\bbox M\iso M^{DA}$ has already been established in Proposition~\ref{pr:box-isos}.
It remains to show $\uHom(F_A,M)\iso M^A$.  For this we observe the sequence of natural isomorphisms
\begin{align*}
\uHom(F_A,M)(X)=\Hom(F_X\bbox F_A,M)=\Hom(F_{X\times A},M)&=M(X\times A)\\
&=M(A\times X)\\
&=M^A(X).
\end{align*}
\end{proof}


\bibliographystyle{amsalpha}

\begin{thebibliography}{JTTW}

\bibitem[A]{A} J.~E. Arnold, Jr., {\em Homological algebra based on permutation modules\/}, J. Algebra {\bf 70} (1981), 250--260.   

\bibitem[BD]{BD} S. Basu and P. Dey, {\em Equivariant cohomology for cyclic groups\/}, 
Int. Math. Res. Not. IMRN 2025, Paper No. rnaf150, 41pp.

\bibitem[B1]{B1} S. Bouc, {\em Green functors and $G$-sets\/}, Lecture Notes in Math. {\bf 1671}, Springer-Verlag, Berlin, 1997.

\bibitem[B2]{B2} S. Bouc, {\em Complexity and cohomology of cohomological Mackey functors\/}, Adv. Math. {\bf 221} (2009), no. 3, 983--1045. 

\bibitem[BSW]{BSW} S. Bouc, R. Stancu, and P. Webb, {\em On the projective dimension of Mackey functors\/}, Algebr. Represent. Theory {\bf 20} (2017), no. 6, 1467--1481.

\bibitem[DH1]{DH1} D. Dugger and C. Hazel, {\em $RO(G)$-graded cohomology of equivariant configuration spaces\/}, 
Algebr. Geom. Topol. {\bf 26} (2026), no. 2, 437--484.

\bibitem[DH2]{DH2} D. Dugger and C. Hazel, {\em The Bredon equivariant cohomology of a point for cyclic groups\/}, preprint, 2026.

\bibitem[HHR]{HHR2} M.~A. Hill, M.~J. Hopkins, and D.~C. Ravenel, {\em The slice spectral sequence for certain $RO(C_{p^n})$-graded suspensions of $H\mZ$\/}, Bol. Soc. Mat. Mex. (3) {\bf 23} (2017), no. 1, 289--317.

\bibitem[L1]{L1} L.~G. Lewis, {\em The theory of Green functors\/}, unpublished preprint, 1981.

\bibitem[L2]{L2} L.~G. Lewis, {\em The $RO(G)$-graded equivariant ordinary cohomology of complex projective spaces with linear $\Z/p$-actions\/},
Algebraic topology and transformation groups (G\"ottingen, 1987), 53-122.  Lecture Notes in Math. {\bf 1361}, Springer-Verlag Berlin, 1988.

\bibitem[MQS]{MQS} D. Mehrle, J.~D. Quigley, M. Stahlhauer, {\em Pathological computations of Mackey functor-valued Tor over cyclic groups\/}, Bull. Lond. Math. Soc. {\bf 57} (2025), no. 10, 3024--3036.

\bibitem[TW]{TW} J. Th\'evenaz and P.~J. Webb, {\em The structure of Mackey functors\/}, Trans. Amer. Math. Soc. {\bf 347} (1995), 1865--1961.

\bibitem[W]{W} P.~J. Webb, {\em A guide to Mackey functors\/}, in Handbook of Algebra Vol. 2, Elsevier/North-Holland, Amsterdam, 2000, pp. 805--836.

\bibitem[Y]{Y} T. Yoshida, {\em On $G$-functors (II): Hecke operators and $G$-functors\/}, J. Math. Soc. Japan {\bf 35} (1983), 179--190.

\bibitem[Z]{Z} M. Zeng, {\em Equivariant Eilenberg-MacLane spectra in cyclic $p$-groups\/}, preprint, 2017, arXiv:1710.01769.  (Note that a previous version existed under a different title: {\em $RO(G)$-graded homotopy Mackey functor of $H\mZ$ for $C_{p^2}$ and homological algebra over $\mZ$-modules}).  


\end{thebibliography}

\end{document}